\documentclass[12pt,a4paper]{article}
\usepackage[cp1250]{inputenc}
\usepackage[T1]{fontenc}
\usepackage[dvips]{graphicx}
\usepackage{epsf}
\usepackage{amssymb}
\usepackage{pstricks}
\usepackage{pst-node}
\usepackage{lineno}
\usepackage[normalem]{ulem}
\usepackage{mathrsfs}
\usepackage{amsmath}
\usepackage{amsthm}
\usepackage{amsfonts}
\usepackage{eepic}
\usepackage{color}
\usepackage{amsthm}

\newtheorem{thm}{Theorem}[section]

\newtheorem{observation}[thm]{Observation}
\newtheorem{proposition}[thm]{Proposition}
\newtheorem{cor}[thm]{Corollary}

\newcommand{\2}{ \vspace{0.2cm} }

\newcommand{\cL}{{\cal L}}

\let\oldenumerate\enumerate
\renewcommand{\enumerate}{
  \oldenumerate
  \setlength{\itemsep}{0.5pt}
  \setlength{\parskip}{0pt}
  \setlength{\parsep}{0pt}
}

\begin{document}

\title{Minimal graphs with disjoint dominating\\ and paired-dominating sets}

\author{$^{1}$Michael A. Henning\thanks{Research
supported in part by the University of Johannesburg} \, and \, $^{2}$Jerzy Topp\\
\\
$^1$Department of Mathematics and Applied Mathematics\\
University of Johannesburg \\
Auckland Park 2006, South Africa \\
\small {\tt Email: mahenning@uj.ac.za} \\
\\
$^2$The State University of Applied Sciences in Elbląg, Poland\\
\small {\tt Email: jtopp@inf.ug.edu.pl} }

\date{} \maketitle

\begin{abstract}
\noindent A subset $D\subseteq V_G$ is a dominating set of $G$ if every vertex in $V_G-D$ has a~neighbor in $D$, while $D$ is a paired-dominating set of $G$ if $D$ is a~dominating set and the subgraph induced by $D$ contains a perfect matching. A graph $G$ is a $D\!P\!D\!P$-graph if it has a pair $(D,P)$ of disjoint sets of vertices of $G$ such that $D$ is a dominating set and $P$ is a paired-dominating set of $G$. The study of the $D\!P\!D\!P$-graphs was initiated by Southey and Henning (Cent. Eur. J. Math. 8 (2010) 459--467; J. Comb. Optim. 22 (2011) 217--234). In this paper, we provide conditions which ensure that a graph is a $D\!P\!D\!P$-graph. In particular, we characterize the minimal $D\!P\!D\!P$-graphs.\\

{\bf \AmS \; Subject Classification:} 05C69, 05C85
\end{abstract}

\section{Introduction}

Let $G=(V_G,E_G)$ be a graph with vertex set $V(G) = V_G$ and edge set $E(G) = E_G$, where we allow multiple edges and loops. A set of vertices $D \subseteq V_G$ is a \emph{dominating set} of $G$ if every vertex in $V_G \setminus D$  has a~neighbor in $D$, while $D$ is \emph{$2$-dominating set} of $G$ if every vertex in $V_G \setminus D$  has at least two neighbors in $D$. A set $D \subseteq V_G$ is a~\emph{total dominating set}  of $G$ if every vertex has a~neighbor in $D$. A set $D \subseteq V_G$ is a~\emph{paired-dominating set} of $G$ if $D$ is a~dominating set and the subgraph induced by $D$ contains a perfect matching.

Ore \cite{Ore} was the first to observe that a graph with no isolated vertex contains two disjoint dominating sets. Consequently, the vertex set of a graph without isolated vertices can be partitioned into two dominating sets. Various graph theoretic properties and parameters of graphs having disjoint dominating sets are studied in \cite{AK,HHLMS,HT,HLR09,HM18,KS,LR10}. Characterizations of graphs with disjoint dominating and total dominating sets are given in \cite{HLR10+,HLR10,HLR10++,HS08,HS09,KJ,SH11a}, while in \cite{BDGHHU,DDH,DHH17,DGHH,HY} graphs which have the property that their vertex set can be partitioned into two disjoint total dominating sets are studied.  Conditions which guarantee the existence of a~dominating set whose complement contains a $2$-dominating set, a paired-dominating set or an independent dominating set are presented in~\cite{HaynesHenning2005,HLR10,HR13,KJ,KS,MTZ19,SH11b}.

In this paper we restrict our attention to conditions which ensure a partition of vertex set of a graph into a~dominating set and a paired-dominating set. The study of graphs having a dominating set whose complement is a paired-dominating set was initiated by Southey and Henning  \cite{SH10a,SH11b}. They define a \emph{$D\!P$-pair} in a~graph $G$ to be a~pair $(D,P)$ of disjoint sets of vertices of $G$ such that $V(G) = D \cup P$ where $D$ is a dominating set and $P$ is a~paired-dominating set of $G$. A graph that has a~$D\!P$-pair is called a~\emph{$D\!P\!D\!P$-graph} (standing, as in \cite{SH10a,SH11b}, for ``dominating, paired dominating, partitionable graph").  It is easy to observe that a~complete graph $K_n$ is a~$D\!P\!D\!P$-graph if $n\ge 3$ (and $K_3$ is the smallest $D\!P\!D\!P$-graph), a~path $P_n$ is a~$D\!P\!D\!P$-graph if and only if $n\in \mathbb{N} \setminus \{1, 2, 3, 5, 6, 9\}$, while a~cycle $C_n$ is a~$D\!P\!D\!P$-graph if $n\ge 3$ and $n\ne 5$. It was also proved in \cite{SH10a} that every cubic graph is a~$D\!P\!D\!P$-graph. In \cite{SH11b} the $D\!P\!D\!P$-graphs (and, in particular, the $D\!P\!D\!P$-trees) were characterized as the graphs which can be constructed from a labeled $P_4$ by applying eight (four, resp.) operations.

For notation and graph theory terminology we in general follow~\cite{ChLZ16}. Specifically, for a vertex $v$ of a~graph $G=(V_G,E_G)$, its \emph{neighbourhood\/}, denoted by $N_{G}(v)$, is the set of all vertices adjacent to $v$, and the cardinality of $N_G(v)$, denoted by $d_G(v)$, is called the \emph{degree} of~$v$. The \emph{closed neighbourhood\/} of $v$, denoted by $N_{G}[v]$, is the set $N_{G}(v)\cup \{v\}$. In general, for a subset $X\subseteq V_G$ of vertices, the \emph{neighbourhood\/} of $X$, denoted by $N_{G}(X)$, is defined to be $\bigcup_{v\in X}N_{G}(v)$, and the \emph{closed\/} neighborhood of $X$, denoted by $N_{G}[X]$, is the set $N_{G}(X)\cup X$. The minimum degree of a vertex in $G$ is denoted by $\delta(G)$.  A vertex of degree one is called a \emph{leaf}, and the only neighbor of a~leaf is called its \emph{support vertex} (or simply, its \emph{support}). If a support vertex has at least two leaves as neighbors,  we call it a \emph{strong\/} support, otherwise it is a \emph{weak} support. The set of leaves, the set of weak supports, the set of strong supports, and the set of all supports of $G$ is denoted by $L_G$, $S'_G$, $S''_G$,  and $S_G$, respectively. If $v$ is a~vertex of $G$, then by $E_G(v)$ and $\cL_G(v)$ we denote the set of edges and the set of loops incident with $v$ in $G$, respectively.

We denote the \emph{path}, \emph{cycle}, and \emph{complete graph} on $n$ vertices by $P_n$, $C_n$, and $K_n$, respectively. The \emph{complete bipartite graph} with one partite set of size~$n$ and the other of size~$m$ is denoted by $K_{n,m}$. A \emph{star} is the tree $K_{1,k}$ for some $k \ge 1$. For $r, s \ge 1$, a \emph{double star} $S(r,s)$ is the tree with exactly two vertices that are not leaves, one of which has $r$ leaf neighbors and the other $s$ leaf neighbors. We define a \emph{pendant edge} of a graph to be an edge incident with a vertex of degree~$1$. We use the standard notation $[k] = \{1,\ldots,k\}$.

\section{$2$-Subdivision graphs of a graph}

Let $H=(V_H,E_H)$ be a graph with no isolated vertices and with possible multi-edges and multi-loops. By $\varphi_H$ we denote a function from $E_H$ to $2^{V_H}$ that associates with each $e\in E_H$ the set $\varphi_H(e)$ of vertices incident with $e$. Let $X_2$ be a set of $2$-element subsets of an arbitrary set (disjoint with $V_H \cup E_H$), and let $\xi  \colon E_H \to X_2$ be a function such that $\xi(e)\cap \xi(f) = \emptyset$ if $e$ and $f$ are distinct elements of $E_H$. If $e\in E_H$ and $\varphi_H(e)=\{u,v\}$ ($\varphi_H(e)=\{v\}$, respectively), then we write $\xi(e)= \{u_e,v_e\}$ ($\xi(e)= \{v^1_e,v^2_e\}$, respectively). If $\alpha \colon L_H\to \mathbb{N}$ is a function, then let $\Phi_\alpha \colon L_H \to L_H\times \mathbb{N}$ be a~function such that $\Phi_\alpha(v)=\{(v,i)\colon i \in [ \alpha(v)] \}$ for $v\in L_H$.

Now we say that a graph $S_2(H)= (V_{S_2(H)},E_{S_2(H)})$ is the $2$-\emph{subdivision graph} of $H$ (with respect to the functions $\xi\colon E_H\to X_2$ and $\alpha \colon L_H \to \mathbb{N}$), if $V_{S_2(H)}=V_{S_2(H)}^o\cup V_{S_2(H)}^n$, where
\[
V_{S_2(H)}^o = (V_H \setminus L_H)\cup \bigcup_{v\in L_H}\Phi_\alpha(v) \hspace*{0.5cm} \mbox{and} \hspace*{0.5cm} V_{S_2(H)}^n =  \bigcup_{e\in E_H}\xi(e),
\]
and where
\[
\begin{array}{lcl}
E_{S_2(H)} & = & \displaystyle{ \bigcup_{e\in E_H} \{xy\colon
\xi(e)=\{x,y\}\} \, \cup } \2 \\
& & \hspace*{0.3cm} \displaystyle{  \bigcup_{v\in V_H \setminus L_H} \left(\{vv_e \colon e \in E_H(v)\} \cup  \{vv^1_e, vv^2_e \colon e\in \cL_H(v)\} \right) \, \cup } \2 \\
& & \hspace*{1.5cm} \displaystyle{  \bigcup_{v\in L_H}\{v_e(v,i) \colon e\in E_H(v),\,\, i \in [\alpha(v)]\}.}
\end{array}
\]

\section{Main Results}

In this paper, our aim is to characterize $D\!P\!D\!P$-graphs. The following result provided a characterization of minimal $D\!P\!D\!P$-graphs, where a good subgraph is defined in Section~\ref{S:good}.


\begin{thm}
\label{thm:main1} If $G$ is a connected graph of order at least three, then the following statements are equivalent:
\begin{enumerate}
\item[$(1)$] $G$ a minimal $D\!P\!D\!P$-graph.
\item[$(2)$] $G=S_2(H)$ for some connected graph $H$, and either $(V_{S_2(H)}^o,V_{S_2(H)}^n)$ is the unique $D\!P$-pair in $G$ or $G$ is a cycle of length 3, 6 or 9.
\item[$(3)$] $G=S_2(H)$ for some connected graph $H$ that has neither an isolated vertex nor a good subgraph.
\item[$(4)$] $G=S_2(H)$ for some connected graph $H$ and no proper spanning subgraph of $G$ without isolated vertices is a $2$-subdivided graph.
\end{enumerate}
\end{thm}

\section{Properties of $2$-subdivision graphs}

We remark that $2$-subdivision graphs are defined only for graphs without isolated vertices and, intuitively, $S_2(H)$ is the graph obtained from $H$ by inserting two new vertices into each edge and each loop of H, and then replacing each pendant edge $v_ev$ by pendant edges $v_e(v,1), \ldots, v_e(v,\alpha(v))$. In particular, it follows from this definition that every tree of diameter three (i.e., every double star) is a $2$-subdivision graph of $K_2$. Moreover, a path $P_n$ (of order $n$) is a $2$-subdivision graph (of a path) if and only if $n=3k+1$ for every positive integer $k$ and here $\alpha$ assign to each leaf the value~$1$. Fig. \ref{2-subdivision-graph} shows a~graph $H$ and a possible $2$-subdivision graph $S_2(H)$ of $H$ where here $\alpha \colon L_H \to \{3\}$.

\begin{figure}[h!]\begin{center}  \special{em:linewidth 0.4pt} \unitlength 0.4ex \linethickness{0.4pt} \begin{picture}(220,55)\put(30,10){\path(0,30)(0,0)(30,30) \path(30,0)(60,0)\bezier{200}(0,0)(15,-12)(30,0)\bezier{200}(0,0)(15,12)(30,0)
\path(0,30)(30,30)\put(-4.0,-3){${}^{s}$}\put(-4,27){${}^{r}$}\put(29.5,30){$^t$}\put(28.95,-5.5){$^u$}
\put(-4,12){$^a$}\put(12,13){$^c$}\put(14,29){$^b$}\put(14,5.0){$^d$}\put(14,-7.0){$^e$}
\put(43.0,-6.0){$^f$}\put(58.5,-6.0){$^v$}\put(47,27){$^g$}\put(29.0,9.5){$^h$}
\put(30,30){\bezier{200}(0,0)(19,-12)(20,0)\bezier{200}(0,0)(19,12)(20,0)}
\put(30,30){\bezier{200}(0,0)(-12,-19)(0,-20)\bezier{200}(0,0)(12,-19)(0,-20)}
\multiput(0,0)(30,0){3}{\circle*{2}}\multiput(0,30)(30,0){2}{\circle*{2}}\put(-7,37){$^H$}}
\put(110,10){\path(0,30)(0,0)(30,30)\path(30,0)(60,0)\path(0,30)(30,30)
\bezier{200}(0,0)(15,-12)(30,0)\bezier{200}(0,0)(15,12)(30,0)
\put(30,30){\bezier{200}(0,0)(19,-12)(20,0)\bezier{200}(0,0)(19,12)(20,0)}
\put(30,30){\bezier{200}(0,0)(-12,-19)(0,-20)\bezier{200}(0,0)(12,-19)(0,-20)}
\multiput(0,0)(30,0){2}{\circle*{2}}\multiput(0,30)(30,0){2}{\circle*{2}}
\multiput(0,10)(0,10){2}{\whiten\circle{2}}\put(-4.0,-3){${}^{s}$}\put(-4,27){${}^{r}$} \put(29.5,30){$^t$}\put(28.95,-5.5){$^u$}\put(-5.5,7){$^{s_a}$}\put(-5.5,17){$^{r_a}$} \put(8.5,30.5){$^{r_b}$}\put(18.5,30.5){$^{t_b}$}\put(4.5,8){$^{s_c}$}\put(15.0,18.0){$^{t_c}$}
\put(8.0,-0.5){$^{s_d}$}\put(18,-0.5){$^{u_d}$}\put(8.0,-5){$^{s_e}$}\put(18,-5){$^{u_e}$}
\put(25.5,12.0){$^{t^1_h}$}\put(37.5,12.0){$^{t^2_h}$}\put(37.5,-6){$^{u_f}$}\put(47.5,-6){$^{v_f}$}
\put(43.5,16.5){$^{t^1_g}$}\put(43.5,37.5){$^{t^2_g}$}\put(60,-9){$^{(v,1)}$}\put(62,-3.5){$^{(v,2)}$}
\put(60,2.5){$^{(v,3)}$}\multiput(10,10)(10,10){2}{\whiten\circle{2}}
\multiput(10,-5.4)(10,0){2}{\whiten\circle{2}}\multiput(10,5.4)(10,0){2}{\whiten\circle{2}}
\multiput(10,30)(10,0){2}{\whiten\circle{2}}\multiput(45,24)(0,12){2}{\whiten\circle{2}}
\multiput(24,15)(12,0){2}{\whiten\circle{2}}\path(58,4)(50,0)(58,-4)
\multiput(60,0)(0,0){1}{\circle*{2}}\multiput(58,-4)(0,8){2}{\circle*{2}}
\multiput(40,0)(10,0){2}{\whiten\circle{2}}\put(-7,37){$^{S_2(H)}$}}
\end{picture}\caption{A $2$-subdivision graph $S_2(H)$ of a graph $H$} \label{2-subdivision-graph}\end{center}\end{figure}

\begin{observation} \label{pierwsze-wlasnosci-S2(H)}
Let $H$ be a graph with no isolated vertex, and let $G=S_2(H)$ be the $2$-subdivision graph of $H$ $($with respect to functions $\xi\colon E_H\to X_2$ and $\alpha \colon L_H \to \mathbb{N}$$)$. Then the following statements hold.
\begin{enumerate}
\item[$(1)$] $d_{G}(v)= d_H(v)$ if $v\in V_H \setminus L_H$, and $d_{G}((v,i))= 1$ if $v\in L_H$ and $i \in [\alpha(v)]$.
\item[$(2)$] $d_G(x)=2$ if $x\in V_{S_2(H)}^n \setminus S_G$, and $d_G(v_e)=1+\alpha(v)$ if $v\in L_H$ and $e\in E_H(v)$.
\item[$(3)$] If $x, y\in V_G \setminus V_{S_2(H)}^n$ are distinct, and belong to the same component of $G$, then either $d_G(x,y) \equiv 0 \,(\hspace{-1.5ex}\mod 3)$ or $x,y\in L_G$ and $d_G(x,y)=2$.
\item[$(4)$] If $x\in S_G$, then $|N_G(x)\cap V_{S_2(H)}^n| = 1$ and $N_G(x) \setminus V_{S_2(H)}^n \subseteq L_G$.
\item[$(5)$] If $x \in V_G$, then the following hold. \\ [-20pt]
\begin{enumerate}
\item[{\rm (a)}] If $d_G(x)>2$, then either $x\in V_H$ or $x\in S_G$ and $|N_G(x) \setminus L_G|=1$.
\item[{\rm (b)}]  If $x\in V_{S_2(H)}^n$, then either $d_G(x)=2$ or $d_G(x)>2$ and $x\in S_G$.
\item[{\rm (c)}] If $x\in V_{S_2(H)}^n$ and $d_G(x)>2$, then $x\in S_G$.
\end{enumerate}
\item[$(6)$] Let $G'$ be a $2$-subdivision graph which is a spanning subgraph of $G$. If $F$ is a~component of $G'$, then $F$ has exactly one of the following properties. \\ [-20pt]
\begin{enumerate}
\item[{\rm (a)}] $F$ is an induced subgraph of $G$ if no leaf of $F$ is in $V_{S_2(H)}^n$.
\item[{\rm (b)}] $F$ is a $2$-subdivision graph of a~path $P_{k+1}$ $($$k\ge 1$$)$ and $F$ has at most one strong support vertex if at least one leaf of $F$ is in $V_{S_2(H)}^n$. In addition, exactly one of these support vertices is in $V_H$. Moreover, if $F$ has a strong support vertex, then this strong support vertex is in $V_H$.
\end{enumerate}
\end{enumerate}
\end{observation}

\begin{proof} The statements (1)--(5) are immediate consequences of the definition of the $2$-subdivision graph. To prove (6), let $G'$ be a spanning subgraph of $G$ that is a~$2$-subdivision graph and let $F$ be a component of $G'$. Since $G'$ is a $2$-subdivision graph, so too is the graph $F$, i.e., $F=S_2(H')$ for some connected graph $H'$ (and some functions $\xi'\colon E_{H'}\to X_2$ and $\alpha' \colon L_{H'} \to \mathbb{N}$).

\medskip
\emph{Case 1. $L_F \cap V_{S_2(H)}^n = \emptyset$.}  Since $F$ is a $2$-subdivision graph, the sets $V_{F}\cap V_H$ and $V_{F}\cap V_{S_2(H)}^n$ are nonempty. Assume first that $v\in V_{F}\cap V_H$ and $e$ is a loop at $v$ in $G$. We claim that the vertices $v_e^1$ and $v_e^2$, and the edges $vv_e^1$, $v_e^1v_e^2$, $vv_e^2$ belong to $F$. If $v_e^1$ or $v_e^2$ were not in $F$, then $G'$ (which is a spanning subgraph of $G$) would have a component of order one or two, which is impossible in a $2$-subdivision graph. Now, since neither $v_e^1$ nor $v_e^2$ is a leaf in $F$, both $v_e^1$ and $v_e^2$ are of degree~$2$ in $F$ and this proves that the edges $vv_e^1$, $v_e^1v_e^2$, $vv_e^2$ belong to $F$. We can similarly show that if $u, v\in V_{F}\cap V_H$ and $e$ is an edge joining $u$ to $v$ in $G$, then the vertices $u_e$ and $v_e$, and the edges $vv_e$, $v_eu_e$, $u_eu$ belong to $F$. From this it follows that $F$ is a $2$-subdivision graph of the induced subgraph $H[V_{F}\cap V_H]$ and, therefore, $F$ is an induced subgraph of $G$.

\medskip
\emph{Case 2. $L_F \cap V_{S_2(H)}^n \ne \emptyset$.} Let $x_0$ be a leaf of $F$ which belongs to $V_{S_2(H)}^n$ in $G$. Since $G'$ is a $2$-subdivision graph, the vertex $x_0$ does not belong to $N_G[S_G]$. In addition, if $\Delta(F)\le 2$, then $F$ is a path, and, since $F$ is a $2$-subdivision graph, we note that $F=P_{3k+1}=S_2(P_{k+1})$ (for some positive integer $k$), as desired. Thus assume that $\Delta(F)\ge 3$. Let $x$ be a vertex of degree at least 3 in $F$. It follows from (5) applied to the graph $F=S_2(H')$ that either $x\in V_{H'}$ or $x\in S_F$ and $|N_F(x) \setminus L_F|=1$. However, such a vertex $x$ cannot be in $V_{H'} =\{y\in V_F \colon d_F(x_0,y)\equiv 0 \,(\hspace{-1.5ex} \mod 3)\}$, as every vertex belonging to $V_{H'} \setminus L_{H'} \subseteq V_{S_2(H)}^n$ is of degree~$2$ in $G$ and in $F$, while vertices in $F$ corresponding to elements of $L_{H'}$ are of degree~$1$. This proves that $x\in S_F$ and $|N_F(x) \setminus L_F|=1$, that is, every vertex of degree at least~$3$ in $F$ is a strong support vertex and it has only one neighbor which is not a leaf. From this it follows that $F$ is a $2$-subdivision graph $S_2(P_{k+1})$ with at least one strong support vertex (for some positive integer $k$).

It remains to show that $F$ cannot have two strong support vertices. Suppose, for the sake of contradiction, that $s_1$ and $s_2$ are distinct strong support vertices in $F$. Let $\ell_1$ and $\ell_2$ be leaves in $F$ adjacent to $s_1$ and $s_2$, respectively. Since $s_1$ and $s_2$ are vertices of degree at least three in $G=S_2(H)$, it follows from (5) that each of them belongs to $V_H$ or $S_G$. There are three cases to consider. If $s_1, s_2 \in V_H$, then it follows from (3) that $d_G(s_1,s_2) \equiv 0 \,(\hspace{-1.5ex} \mod 3)$, implying that $d_F(\ell_1,\ell_2) \equiv 2 \,(\hspace{-1.5ex} \mod 3)$ and $d_F(\ell_1,\ell_2) \ne 2$, contradicting (3) in $F$. Hence renaming $s_1$ and $s_2$ if necessary, we may assume that $s_2 \in S_G$. If $s_1\in V_H$, then it follows from (3) that $d_G(s_1,\ell_2) \equiv 0 \,(\hspace{-1.5ex} \mod 3)$, implying that $d_F(\ell_1,\ell_2) \equiv 1 \,(\hspace{-1.5ex} \mod 3)$, contradiction to (3) in $F$. Hence, $s_1 \in S_G$. Thus, no leaf of $F$ belongs to $V_{S_2(H)}^n$ in $G$,  contradicting our choice of $F$. This completes the proof of the statement~(6).
\end{proof}

We next present the following elementary property of a $D\!P\!D\!P$-graph.

\begin{observation}
\label{observ-1}
If $(D,P)$ is a $D\!P$-pair in a graph $G$, then every leaf of $G$ belongs to $D$, while every support of $G$ is in $P$, that is, $L_G\subseteq D$ and $S_G\subseteq P$.
\end{observation}

A connected graph $G$ is said to be a \emph{minimal $D\!P\!D\!P$-graph}, if $G$ is a~$D\!P\!D\!P$-graph and no proper spanning subgraph of $G$ is a~$D\!P\!D\!P$-graph.

We remark that a complete graph $K_n$ is a minimal $D\!P\!D\!P$-graph only if $n=3$. We observe that a path $P_n$ is a minimal $D\!P\!D\!P$-graph if and only if $n\in \{4, 7, 10, 13\}$, while a cycle $C_n$ is a minimal $D\!P\!D\!P$-graph if and only if $n\in \{3, 6, 9\}$. From the definition of a~minimal $D\!P\!D\!P$-graph we immediately have the following important (and intuitively easy) observation.

\begin{observation}
\label{observ-2-supergraph}
Every spanning supergraph of a~$D\!P\!D\!P$-graph is a $D\!P\!D\!P$-graph, and, trivially, every $D\!P\!D\!P$-graph is a spanning supergraph of some~minimal $D\!P\!D\!P$-graph.
\end{observation}

We show next that the $2$-subdivision graph of an isolate-free graph is a $D\!P\!D\!P$-graph.

\begin{proposition}
\label{prop-S2(H)-jest-DPDP-grafem}
If a graph $H$ has no isolated vertex, then its $2$-subdivision graph $S_2(H)$ is a $D\!P\!D\!P$-graph.
\end{proposition}
\begin{proof}
Let $S_2(H)$ be the subdivision graph of $H$ (with respect to functions $\xi\colon E_H\to X_2$ and $\alpha \colon L_H \to \mathbb{N}$). We shall prove that $(D,P)$ is a $D\!P$-pair in $S_2(H)$, where
\[
D= V_{S_2(H)}^o = (V_H \setminus L_H)\cup \bigcup_{v\in L_H}\Phi_\alpha(v)
\]
and $P = V_{S_2(H)}^n = V_{S_2(H)} \setminus D$. If $x \in P$, then $x\in \xi(e)$ for some $e\in E_H$, and $x$ is adjacent in $S_2(H)$ to a vertex incident with $e$ in $H$. This proves that $D$ is a~dominating set of $S_2(H)$. Assume now that $y \in D$. If $y\in V_H \setminus L_H$, then, since $H$ has no isolated vertex, there is an edge $f$ incident with $y$ in $H$, and therefore $y$ is adjacent to $y_f \in P$ (or to $y^1_f \in P$ and $y^2_f \in P$ if $f$ is a loop) in $S_2(H)$. If $y\in \Phi_\alpha(v)$ for some $v\in L_H$, then $y$ is adjacent to $v_e$ in $S_2(H)$, where $e$ is the only pendant edge incident with $v$ in $H$. Consequently, $P$ is a dominating set of $S_2(H)$. In addition, since the two vertices of $\xi(e)$ are adjacent in $S_2(H)$ for every $e \in E_H$, the set $P= \bigcup_{e\in E_H}\xi(e)$ is  a~paired-dominating set of $S_2(H)$.
This proves that $S_2(H)$ is a~$D\!P\!D\!P$-graph.
\end{proof}

Since every graph is homeomorphic to its $2$-subdivision graph, it follows from Proposition~\ref{prop-S2(H)-jest-DPDP-grafem} that every graph without isolated vertices is homeomorphic to a~$D\!P\!D\!P$-graph.  Consequently, the structure of  $D\!P\!D\!P$-graphs becomes more complex.

The next theorem presents general properties of $D\!P$-pairs in a minimal $D\!P\!D\!P$-graph.

\begin{thm} \label{observ-2-minimal-DPDP-graphs}
If $G$ is a minimal $D\!P\!D\!P$-graph and $(D,P)$ is a $D\!P$-pair in $G$, then the following four statements hold. \begin{enumerate}
\item[$(1)$] $D$ is a~maximal independent set in $G$.
\item[$(2)$] The induced graph $G[P]$ consists of independent edges, that is, $\delta(G[P])=\Delta(G[P])=1$.
\item[$(3)$] If $x\in P$, then $|N_G(x) \setminus P|=1$ or $N_G(x) \setminus P$ is a~nonempty subset of $L_G$.
\item[$(4)$] $G$ is a $2$-subdivision graph of some graph $H$.
\end{enumerate}
\end{thm}
\begin{proof} (1) If $D$ is not an independent set, then $D$ contains two vertices, say $x$ and $y$, that are adjacent. In this case, $(D,P)$ would be a~$D\!P$-pair in $G-xy$,  contradicting the minimality of $G$. Hence, the set $D$ is both an independent and dominating set of $G$, implying that $D$ is a maximal independent set in $G$.

(2) Since $P$ is a paired-dominating set of $G$, by definition, $G[P]$ has a perfect matching, say $M$. If $xy$ is an edge of $G[P]$ which is not in $M$, then $(D,P)$ would be a~$D\!P$-pair in $G-xy$, violating the minimality of $G$. Hence, the edges of $M$ are the only edges of $G[P]$.

(3) Assume that $x \in P$. It follows from (2) that $x$ has exactly one neighbor in $P$, say $x'$. Thus since $(D,P)$ is a $D\!P$-pair in $G$, we note that $N_G(x) \setminus \{x'\}$ is a~nonempty subset of the dominating set $D$ of $G$. If every neighbor of $x$ in $D$ is a leaf, then $N_G(x) \setminus P$ is a nonempty subset of $L_G$. Hence we may assume that $x$ contains a neighbor $y$ in $D$ that is not a leaf, for otherwise the desired result follows. If $x$ contains a neighbor in $D$ different from $y$, then $x$ is dominated by a vertex belonging to $D \setminus \{y\}$ and $y$ is dominated by some vertex in $P \setminus \{x\}$, implying that $(D,P)$ is a~$D\!P$-pair in $G-xy$, contradicting the minimality of $G$. Hence in this case, the vertex $y$ is the only neighbor of $x$ in $D$, and so $N_G(x) = \{x',y\}$ and $|N_G(x) \setminus P| = 1$.

(4) Let $G$ be a~minimal $D\!P\!D\!P$-graph, and let $(D,P)$ be a~$D\!P$-pair in $G$. For a~support vertex $s$, the set of leaves adjacent to $s$ is denoted by $L_G(s)$, i.e., $L_G(s) = N_G(s)\cap L_G$. Let $G^*$ denote the graph resulting from $G$ by replacing the vertices of $L_G(s)$ by a new vertex $v_s$ and joining $v_s$ to $s$, for every $s\in S_G$, i.e., $G^*=(V_{G^*},E_{G^*})$, where $V_{G^*}= (V_G \setminus L_G) \cup \{v_s \colon s\in S_G\}$ and $E_{G^*}= E_{G-L_G}\cup \{sv_s \colon s\in S_G\}$. By (2) and (3) above, we note that every vertex of $P$ has degree~$2$ in $G^*$. Further, every vertex of $P$ has exactly one neighbor in $P$.

We define a graph $H = (V_H,E_H,\varphi_H)$ as follows. Let $V_H = V_{G^*} \setminus P$. For every edge $v_1v_2$ in $G^*$ that joins two vertices of $P$ we do the following. If $v_1$ and $v_2$ have a~common neighbor, say $v$, in $G^*$, then in $H$ we add a loop in $H$ at the vertex~$v$. If $v_1$ and $v_2$ do not have a common neighbor in $G^*$, then we add the edge $u_1u_2$ to $H$ where $u_1$ is the neighbor of $v_1$ different from $v_2$ and where $u_2$ is the neighbor of $v_2$ different from $v_1$ (and so $u_1v_1v_2u_2$ is a path in $G^*$). We let $\varphi_H \colon E_H \to 2^{V_H}$ be the function such that $\varphi_H(m)= N_{G^*}(m) \setminus \{m\}$ if $m\in E_H$. Now let $\xi\colon E_H \to 2^{V_H}$ and $\alpha \colon L_H \to \mathbb{N}$ be functions such that $\xi(e)=e$ if $e\in E_H$, and $\alpha(v_s)= |L_G(s)|$ if $v_s\in L_H$. With these definitions, the graph $G$ is isomorphic to the $2$-subdivision graph $S_2(H)$ of $H$ (with respect to functions $\xi\colon E_H\to 2^{V_H}$ and $\alpha \colon L_H \to \mathbb{N}$). That means that we can restore the graph $G$ by applying the operation $S_2$ to the graph $H$. \end{proof}

By Theorem~\ref{observ-2-minimal-DPDP-graphs}\,(4) every minimal $D\!P\!D\!P$-graph is a $2$-subdivision graph of some graph. The converse, however, is not true in general. For example, if $H$ is the underlying graph of any of the graphs in Fig.~\ref{xxx-graphs}, then its $2$-subdivision graph $S_2(H)$ is a $D\!P\!D\!P$-graph, but it is not a minimal $D\!P\!D\!P$-graph. The following result, which is a special case of Theorem~\ref{S2(H) is minimal DPDP-graph} proven later in the paper, will be useful to establish which $2$-subdivision graphs are not minimal $D\!P\!D\!P$-graphs.

\begin{proposition}
\label{prop:adjdeg2} Let $x$ and $y$ be adjacent vertices of degree 2 in a graph $H$ without isolated vertices, and let $x'$ and $y'$ be the vertices such that $N_H(x) \setminus \{y\}=\{x'\}$ and $N_H(y) \setminus \{x\}=\{y'\}$, respectively. If the sets $N_H(x') \setminus \{x,y\}$ and $N_H(y') \setminus \{x,y\}$ are nonempty, then the $2$-subdivision graph $S_2(H)$ is a $D\!P\!D\!P$-graph but not a minimal $D\!P\!D\!P$-graph.
\end{proposition}
\begin{proof}
It follows from Proposition~\ref{prop-S2(H)-jest-DPDP-grafem} that $S_2(H)$ is a~$D\!P\!D\!P$-graph and $(D,P)$ is a~$D\!P$-pair in $S_2(H)$, where  $D=V_{S_2(H)}^o$ and $P=V_{S_2(H)}^n$. The pair $(D',P')$, where $D'= (D \setminus \{x,y\})\cup \{x_{xy}, y_{xy}, x'_{xx'}, y'_{yy'}\}$ and $P'= (P \setminus \{x'_{xx'}, y'_{yy'}\})\cup \{x,y\}$ is a $D\!P$-pair in the proper spanning subgraph $S_2(H) \setminus \{x_{xy}y_{xy}, x'x'_{xx'}, y'y'_{yy'}\}$ of $S_2(H)$. Thus, $S_2(H)$ is a $D\!P\!D\!P$-graph but not a minimal $D\!P\!D\!P$-graph.
\end{proof}

As a consequence of Proposition~\ref{prop:adjdeg2}, we can readily determine the minimal $D\!P\!D\!P$-paths and minimal $D\!P\!D\!P$-cycles.

\begin{cor}
\label{c:paths_cycles}
The following holds.
\\ [-20pt]
\begin{enumerate}
\item[{\rm (a)}] If $P_n$ is a path of order $n$, then $S_2(P_n)$ is a $D\!P\!D\!P$-graph for every $n\ge 2$, and $S_2(P_n)$ is a minimal  $D\!P\!D\!P$-graph if and only if $n\in \{2, 3, 4, 5\}$.
\item[{\rm (b)}]  If $C_{m}$ is a cycle of size~$m$, then $S_2(C_m)$ is a $D\!P\!D\!P$-graph for every positive integer $m$, and $S_2(C_m)$ is a minimal $D\!P\!D\!P$-graph if and only if $m \in [3]$.
\end{enumerate}
\end{cor}

\section{Good subgraphs of a graph}
\label{S:good}

In this section, we define a good subgraph of a graph. Let $Q$ be a subgraph without isolated vertices of a graph $H$, and let $E_Q^-$ denote the set of edges belonging to $E_H \setminus E_Q$ that are incident with at least one vertex of $Q$. Let $E$ be a set such that $E_Q^-\subseteq E \subseteq E_H \setminus E_Q$, and let $A_E$ is a set of arcs obtained by assigning an orientation for each edge in $E$. Then by  $H(A_E)$ we denote the partially oriented graph obtained from $H$ by replacing the edges in $E$ by the arcs belonging to $A_E$. If $e\in E$, then by $e_A$ we denote the only arc in $A_H$ that corresponds to $e$. By $H_0$ we denote the subgraph of $H(A_E)$ induced by the vertices that are not the initial vertex of an arc belonging to $A_E$, i.e., by the set $\{v\in V_H \colon d^+_{H(A_E)}(v) = 0\}$.

We say that $Q$ is a~\emph{good subgraph of $H$} if there exist a set of edges $E$ (where $E_Q^-\subseteq E \subseteq E_H \setminus E_Q$) and a set of arcs $A_E$ such that in the resulting graph $H(A_E)$, which we simply denote by $H$ for notational convenience, the arcs in $A_E$ form a family ${\cal P}= \{P_x \colon x\in V_Q\}$ of oriented paths indexed by the vertices of $Q$ and such that the following holds.
\\ [-20pt]
\begin{enumerate}
\item[{\rm (1)}] Every vertex of $Q$ is an initial vertex of exactly one path belonging to ${\cal P}$. For each vertex $v \in Q$, we denote the (unique) path belonging to ${\cal P}$ that begins at $v$ by $P_v$. Thus, if $v \in V_Q$, then $d^+_H(v)=1$ and $d^-_H(v)= d_H(v)-d_Q(v)-1$.
\item[{\rm (2)}]  If $x$ is an inner vertex of a path $P_v\in {\cal P}$, then $d^+_H(x)=1$ and $d^-_H(x)=d_H(x)-1$.
\item[{\rm (3)}] If $x$ is a end vertex of a path $P_v\in {\cal P}$, then $d^-_H(x)<d_H(x)$.
\end{enumerate}

Examples of good subgraphs in small graphs are presented in Fig. \ref{xxx-graphs}. For clarity, the edges of a good subgraph $Q$ are drawn in bold, the arcs belonging to oriented paths are thin (and their orientations are represented by arrows), and all other edges, if any, belong to the subgraph $H_0$, are thin and without arrows.

We remark that not every graph has a good subgraph (see also Observation \ref{observation-warm-tree-in-tree} and Corollary \ref{corollary-warm-forest-in-tree}). On the other hand, if $Q$ is a graph with no isolated vertex, and $H$ is the graph obtained from $Q$ by attaching one pendant edge to each vertex of $Q$ and then subdividing this edge, then $Q$ is a good subgraph in $H$, implying that every graph without isolated vertices can be a good subgraph of some graph.

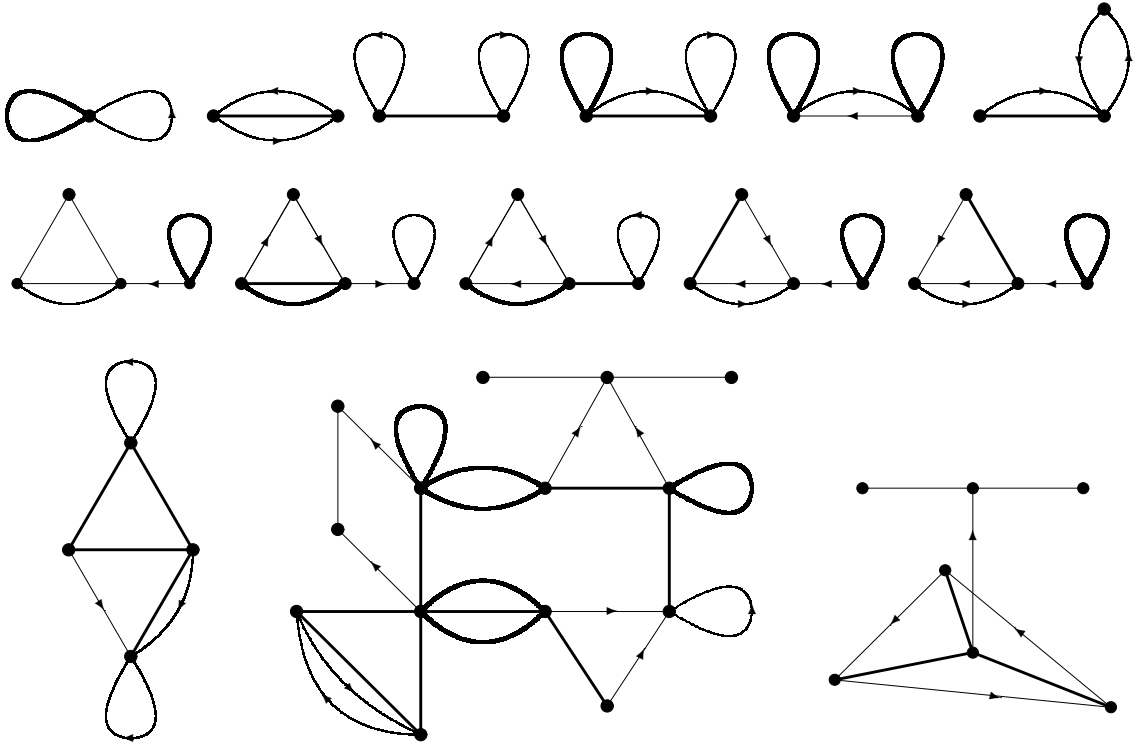
\begin{figure}[h!]\begin{center}  \special{em:linewidth 0.4pt} \unitlength 0.3ex \linethickness{0.4pt} \begin{picture}(220,100)
\put(15,60){\put(-20,10){{\linethickness{0.35mm}\bezier{300}(0,0)(-19,-12)(-20,0)
\bezier{300}(0,0)(-19,12)(-20,0)}{\bezier{300}(0,0)(19,-12)(20,0)\bezier{300}(0,0)(19,12)(20,0)
\put(20,2){\vector(0,1){0.0}}}\put(0,0){\circle*{3}}}%
\put(10,10){{\Thicklines\path(0,0)(30,0)}\bezier{300}(0,0)(15,12)(30,0)\bezier{300}(0,0)(15,-12)(30,0)
\put(13,6){\vector(-1,0){0.0}}\put(17,-6){\vector(1,0){0.0}}\multiput(0,0)(30,0){2}{\circle*{3}}}%
\put(50,10){{\Thicklines\path(0,0)(30,0)}{\bezier{300}(0,0)(-12,19)(0,20) \bezier{300}(0,0)(12,19)(0,20)\bezier{300}(30,0)(18,19)(30,20)\bezier{300}(30,0)(42,19)(30,20)
\put(-1.8,19.8){\vector(-1,0){0.0}}\put(31.8,19.8){\vector(1,0){0.0}}}
\multiput(0,0)(30,0){2}{\circle*{3}}}%
\put(100,10){{\Thicklines\path(0,0)(30,0)\linethickness{0.35mm}
\bezier{300}(0,0)(-12,19)(0,20)\bezier{300}(0,0)(12,19)(0,20)}%
{\bezier{300}(0,0)(15,12)(30,0)\bezier{300}(30,0)(18,19)(30,20) \bezier{300}(30,0)(42,19)(30,20)\put(17,6){\vector(1,0){0.0}}
\put(31.8,19.8){\vector(1,0){0.0}}}\multiput(0,0)(30,0){2}{\circle*{3}}}%
\put(150,10){{\linethickness{0.35mm}\bezier{300}(30,0)(18,19)(30,20)
\bezier{300}(30,0)(42,19)(30,20)\bezier{300}(0,0)(-12,19)(0,20)
\bezier{300}(0,0)(12,19)(0,20)}{\path(0,0)(30,0)\bezier{300}(0,0)(15,12)(30,0)
\put(17,6){\vector(1,0){0.0}}\put(13,0){\vector(-1,0){0.0}}}
\multiput(0,0)(30,0){2}{\circle*{3}}}%
\put(195,10){{\Thicklines\path(0,0)(30,0)}{\bezier{300}(0,0)(15,12)(30,0)
\bezier{300}(30,0)(18,15)(30,26)\bezier{300}(30,0)(42,15)(30,26)
\put(17,6){\vector(1,0){0.0}}\put(36,16){\vector(0,1){0.0}}\put(24,12){\vector(0,-1){0.0}}}
\put(30,26){\circle*{3}}\multiput(0,0)(30,0){2}{\circle*{3}}}}%
%
\special{em:linewidth 0.4pt} \unitlength 0.25ex \linethickness{0.4pt} \begin{picture}(220,40)\put(27,25){\put(-60,10){\path(0,0)(30,0)\path(0,0)(15,26) \path(15,26)(30,0){\path(30,0)(50,0)\put(38,0){\vector(-1,0){0.0}}}
\bezier{300}(0,0)(15,-12)(30,0){\linethickness{0.35mm}\bezier{300}(50,0)(38,19)(50,20)
\bezier{300}(50,0)(62,19)(50,20)}\multiput(0,0)(30,0){2}{\circle*{3}}\put(50,0){\circle*{3}}
\put(15,26){\circle*{3.5}}}
\put(5,10){\path(0,0)(15,26)\path(15,26)(30,0){\path(30,0)(50,0)\put(42,0){\vector(1,0){0.0}} \path(0,0)(15,26)\path(15,26)(30,0)\put(8,14){\vector(1,2){0.0}}\put(23.6,11){\vector(1,-2){0.0}}}%
{\linethickness{0.35mm}\bezier{300}(0,0)(15,-12)(30,0)\Thicklines\path(0,0)(30,0)}%
{\bezier{300}(50,0)(38,19)(50,20)\bezier{300}(50,0)(62,19)(50,20)}
\multiput(0,0)(30,0){2}{\circle*{3.5}}\put(50,0){\circle*{3.5}}\put(15,26){\circle*{3.5}}}
\put(70,10){\path(0,0)(15,26)\path(15,26)(30,0){\path(0,0)(30,0)\bezier{100}(50,0)(38,19)(50,20) \bezier{200}(50,0)(62,19)(50,20)\put(13,0){\vector(-1,0){0.0}}\path(0,0)(15,26)\path(15,26)(30,0)
\put(48,20){\vector(-1,0){0.0}}\put(8,14){\vector(1,2){0.0}}\put(23.6,11){\vector(1,-2){0.0}}}%
{\linethickness{0.35mm}\bezier{100}(0,0)(15,-12)(30,0)\Thicklines\path(30,0)(50,0)} \multiput(0,0)(30,0){2}{\circle*{3.5}}\put(50,0){\circle*{3.5}}\put(15,26){\circle*{3.5}}}
\put(135,10){{\path(30,0)(50,0)\put(38,0){\vector(-1,0){0.0}}\path(15,26)(30,0) \path(0,0)(30,0)\bezier{300}(0,0)(15,-12)(30,0)\put(23.6,11){\vector(1,-2){0.0}} \put(13,0){\vector(-1,0){0.0}}\put(17,-6){\vector(1,0){0.0}}}%
{\Thicklines\path(0,0)(15,26)\linethickness{0.35mm}\bezier{300}(50,0)(38,19)(50,20)
\bezier{300}(50,0)(62,19)(50,20)}\multiput(0,0)(30,0){2}{\circle*{3.5}} \put(50,0){\circle*{3.5}}\put(15,26){\circle*{3.5}}}
\put(200,10){{\path(30,0)(50,0)\put(38,0){\vector(-1,0){0.0}}\path(0,0)(30,0) \bezier{300}(0,0)(15,-12)(30,0)\path(0,0)(15,26)\put(6.6,11){\vector(-1,-2){0.0}}
\put(13,0){\vector(-1,0){0.0}}\put(17,-6){\vector(1,0){0.0}}}%
{\Thicklines\path(15,26)(30,0)\bezier{300}(50,0)(38,19)(50,20)
\bezier{300}(50,0)(62,19)(50,20)}\multiput(0,0)(30,0){2}{\circle*{3.5}}\put(50,0){\circle*{3.5}}
\put(15,26){\circle*{3.5}}}}\end{picture}\end{picture} 
%
\begin{picture}(220,90)
\put(-10,5){{\Thicklines\path(15,24)(30,50)(15,76)(0,50)(30,50)}%
\put(15,76){{\bezier{300}(0,0)(-12,19)(0,20)\bezier{300}(0,0)(12,19)(0,20)
\put(-2,19.8){\vector(-1,0){0.0}}}}{\put(26.3,35){\vector(-1,-2){0.0}}}
\put(8.6,35.2){{\vector(2,-3){0.0}}}{\bezier{300}(15,24)(30,33)(30,50)}
\put(15,24){{\bezier{300}(0,0)(-12,-19)(0,-20)\bezier{300}(0,0)(12,-19)(0,-20)
\put(-2,-19.8){\vector(-1,0){0.0}}\path(0,0)(-15,26)}}
\multiput(15,24)(0,52){2}{\circle*{3}}\multiput(0,50)(30,0){2}{\circle*{3}}}%
\put(45,10){\path(60,30)(90,30)\put(77.5,30){\vector(1,0){0.0}}
{\Thicklines\path(30,0)(30,60) \path(0,30)(60,30)\path(90,30)(90,60)(60,60)
\path(60,30)(75,7)\path(0,30)(30,0)
\put(30,60){\bezier{100}(0,0)(-12,19)(0,20)\bezier{100}(0,0)(12,19)(0,20)}
\put(30,30){\bezier{200}(0,0)(15,-15)(30,0)\bezier{200}(0,0)(15,15)(30,0)}
\put(30,60){\bezier{200}(0,0)(15,-10)(30,0)\bezier{200}(0,0)(15,10)(30,0)}}
\put(10,50){{\path(0,0)(20,-20)\put(8,-8.0){\vector(-1,1){0.0}}}}
\put(10,80){{\path(0,0)(20,-20)\put(8,-8.0){\vector(-1,1){0.0}}}}
\put(60,60){{\path(0,0)(15,27)\put(8.7,15.0){\vector(2,3){0.0}}}}
\put(90,60){{\path(0,0)(-15,27)\put(-8.2,15.0){\vector(-2,3){0.0}}}}
\put(75,7){{\path(0,0)(15,23)\put(9.0,14.0){\vector(1,2){0.0}}}}
\put(90,30){{\bezier{100}(0,0)(19,-12)(20,0)\bezier{100}(0,0)(19,12)(20,0)
\put(20,2){\vector(0,1){0.0}}}}\put(0,30){{\bezier{200}(0,0)(8,-20)(30,-30)
\bezier{200}(0,0)(2,-30)(30,-30)\put(14,-20){\vector(1,-1){0.0}}\put(6,-20){\vector(-1,1){0.0}}}}
\put(90,60){{\Thicklines\bezier{100}(0,0)(19,-12)(20,0)\bezier{100}(0,0)(19,12)(20,0)}}
\multiput(10,50)(0,30){2}{\circle*{3}}\multiput(105,87)(-30,0){3}{\circle*{3}}
\path(10,50)(10,80)\path(45,87)(105,87)\put(75,7){\circle*{3}}
\put(0,0){\multiput(30,0)(0,30){3}{\circle*{3}}\multiput(0,30)(30,0){4}{\circle*{3}}
\multiput(30,60)(30,0){3}{\circle*{3}}}}%
\put(175,10){\unitlength 0.2ex {\Thicklines\path(0,20)(50,30)(100,10)\path(50,30)(40,60)} \path(10,90)(90,90){\path(50,90)(50,30)\path(40,60)(0,20)(100,10)(40,60) \put(50,75){\vector(0,1){0.0}}\put(20,40){\vector(-1,-1){0.0}}\put(65,39){\vector(-3,2){0.0}} \put(60,13.5){\vector(4,-1){0.0}}}\multiput(10,90)(40,0){3}{\circle*{4}}\put(0,20){\circle*{4}}
\put(100,10){\circle*{4}}\put(50,30){\circle*{4}}\put(40,60){\circle*{4}}}%
\end{picture}\caption{Examples of good subgraphs (drawn in bold) in small graphs} \label{xxx-graphs}\label{yyy-graphs}\end{center}\end{figure}

From the definition of a good subgraph we immediately have the following observation.

\begin{observation} \label{observation-leaves-in-paskudny-subgraphs} Neither a leaf nor a support vertex of a graph $H$ belongs to a~good subgraph in $H$. \end{observation}

\section{Structural characterization of $D\!P\!D\!P$-graphs}
\label{S:proof-main1}

In this section, we present a proof of our main result, namely Theorem~\ref{thm:main1}, which provides a characterization of  minimal $D\!P\!D\!P$-graphs. We proceed further with the following result.
	
\begin{thm} \label{propertis-of-minimal}  If $G$ is a connected graph of order at least 3,  then $G$ is a~minimal $D\!P\!D\!P$-graph if and only if $G=S_2(H)$ for some connected graph $H$, and either $(V_{S_2(H)}^o,V_{S_2(H)}^n)$ is the only $D\!P$-pair in $S_2(H)$ or $S_2(H)$ is a cycle of length 3, 6 or 9.
\end{thm}
\begin{proof} If $G=S_2(H)$ is a cycle of length $3$, $6$ or $9$, then $G$ is clearly a~minimal $D\!P\!D\!P$-graph, as claimed. Thus assume that $(V_{S_2(H)}^o,V_{S_2(H)}^n)$ is the only $D\!P$-pair in $G=S_2(H)$. Certainly, $G$ is a $D\!P\!D\!P$-graph, and we shall prove that $G$ is a~minimal $D\!P\!D\!P$-graph. Suppose, to the contrary, that $G$ is not a minimal $D\!P\!D\!P$-graph. Then some proper spanning subgraph $G'$ of $G$ is a  $D\!P\!D\!P$-graph. Let $(D',P')$ be a $D\!P$-pair in $G'$ and, consequently, in $G$ (by Observation \ref{observ-2-supergraph}). Thus $(V_{S_2(H)}^o,V_{S_2(H)}^n)$ and $(D',P')$ are $D\!P$-pairs in $G$, and $(V_{S_2(H)}^o,V_{S_2(H)}^n) \ne (D',P')$, noting that $(V_{S_2(H)}^o,V_{S_2(H)}^n)$ is a $D\!P$-pair in no proper  spanning subgraph of $G=S_2(H)$. This contradicts the uniqueness of a $D\!P$-pair in $G$  and proves that $G$ is a minimal $D\!P\!D\!P$-graph.

Suppose next that $G$ is a minimal $D\!P\!D\!P$-graph. By Theorem \ref{observ-2-minimal-DPDP-graphs}, $G$ is a~$2$-subdivision graph of some connected graph $H$, i.e.,  $G=S_2(H)$, and the pair $(D,P)=(V_{S_2(H)}^o,V_{S_2(H)}^n)$ is a $D\!P$-pair in $S_2(H)$. It remains to prove that either $(V_{S_2(H)}^o,V_{S_2(H)}^n)$ is the only $D\!P$-pair in $S_2(H)$ or $S_2(H)$ is a cycle of length $3$, $6$ or $9$. We consider three cases depending on $\Delta(H)$.

\emph{Case 1. $\Delta(H)=1$.} In this case, $H=P_2$, and its $2$-subdivision graph $S_2(P_2)$ (which is a double star $S(r,s)$ for some positive integers $r$ and $s$) has the desired property.

\emph{Case 2. $\Delta(H)=2$.}  In this case, $H$ is a cycle $C_m$ where $m\ge 1$ or a path $P_n$ where $n\ge 3$. Now, since $S_2(H)$ is a minimal $D\!P\!D\!P$-graph, Corollary~\ref{c:paths_cycles}, implies that $H=C_m$ and $m \in [3]$, or $H=P_n$ and $n\in\{3, 4, 5\}$. In each of these six cases $S_2(H)$ has the desired property.

\emph{Case 3. $\Delta(H) \ge 3$.} In this case, we claim that $(D,P)= (V_{S_2(H)}^o, V_{S_2(H)}^n)$ is the only $D\!P$-pair in $S_2(H)$. Suppose to the contrary that $(D',P')$ is another $D\!P$-pair in $G$. Then, since $D$ and $D'$ are maximal independent sets in $G$ (by Theorem \ref{observ-2-minimal-DPDP-graphs}) and $D\ne D'$, each of the sets $D \setminus D'$ and $D' \setminus D$ is a nonempty subset of $P'$ and $P$, respectively. Let $v$ be a vertex of maximum degree among all vertices in $D \setminus D' \subseteq P'$. Since $v \in P'$, it follows from Theorem \ref{observ-2-minimal-DPDP-graphs} that $d_H(v)\ge 2$. We deal with the two cases when $d_H(v)=2$ and $d_H(v) \ge 3$ in turn.

\emph{Case 3.1. $d_H(v)\ge 3$.} We distinguish three subcases.

\emph{Subcase 3.1.1. There are only loops at $v$ in $H$.} Since $d_H(v)\ge 3$, there are at least two loops at $v$, say $e$ and $f$. Renaming loops if necessary, we may assume that $v^1_e$ is the (unique) neighbor of $v$ belonging to $P'$. We note that $v^2_e \in D'$ and that all other neighbors of $v$ in $G$, including $v^1_f$ and $v^2_f$, belong to $D'$. Therefore, $(D',P')$ is also a~$D\!P$-pair in the proper subgraph $G-vv^2_e$ of $G$, contradicting the minimality of $G$.

\emph{Subcase 3.1.2. There is exactly one loop at $v$ in $H$.} Let $e$ be the loop at $v$ in $H$ and let $f$ be an edge of $H$ incident with $v$. If $v^1_e$ ($v^2_e$, resp.) is the (unique)  neighbor of $v$ belonging to $P'$, then as in Subcase 3.1.1 we infer that $(D',P')$ is a~$D\!P$-pair in the subgraph $G-vv^2_e$ ($G-vv^1_e$, resp.) of $G$. If $v_{f}$ is the (unique)  neighbor of $v$ belonging to $P'$, then $(D',P')$ is a~$D\!P$-pair in the subgraph $G-v^1_ev^2_e$ of $G$. In both cases we get a contradiction to the minimality of $G$.

\emph{Subcase 3.1.3. There is no loop at $v$ in $H$.} In this case, there are three distinct edges, say $e$, $f$, and $g$, incident with $v$ joining $v$ to $u$, $w$, and $z$, respectively. Assume first that $u$, $w$, and $z$ are distinct and, without loss of generality,  $v_{e}$ is the (unique) neighbor of $v$ which belongs to $P'$. Then, since $G$ is a minimal $D\!P\!D\!P$-graph and $(D',P')$ is a~$D\!P$-pair in $G$, Theorem \ref{observ-2-minimal-DPDP-graphs} implies that the vertices $u_{e}$, $v_{f}$, $v_{g}$ belong to $D'$, while $u$, $w$, $w_{f}$, $z$, and $z_{g}$ belong to $P'$. This implies that $(D',P')$ is a~$D\!P$-pair in $G-vv_{g}$, contradicting the minimality of $G$. We derive similar contradictions if $u$, $w$, and $z$ are not distinct, and one of the vertices $v_e$, $v_f$, $v_g$ is the (unique) neighbor of $v$ that belongs to $P'$. We omit the proofs of these cases which are analogous to the previous case when  $u$, $w$, and $z$ are distinct.

\emph{Case 3.2. $d_H(v) =2$.} By our choice of the vertex~$v$, this implies that $d_H(x)=2$ for every $x \in D \setminus D'$. Since $\Delta(H) \ge 3$, we note that $H$ is not a cycle, implying that there is no loop at $v$. Let $e$ and $f$ be the two edges incident with $v$. Renaming the edges $e$ and $f$ if necessary, we may assume that $v_{e}$ is the (unique) neighbor of $v$ in $P'$.

Suppose that $e$ and $f$ are parallel edges. Let $u$ be the second common vertex of $e$ and $f$. In this case, we note that $d_H(u) \ge 3$ as $H$ is not a~cycle. Since $G$ is a minimal $D\!P\!D\!P$-graph and $(D',P')$ is a~$D\!P$-pair in $G$, Theorem~\ref{observ-2-minimal-DPDP-graphs} implies that the vertices $u_{e}$ and $v_{f}$ belong to $D'$, while $u$ and $u_{f}$ belong to $P'$. In particular, $u \in P'$, $d_H(u) \ge 3$, and $u_e$ is a neighbor of $u$ not in $P'$ of degree~$2$. This  contradicts Theorem~\ref{observ-2-minimal-DPDP-graphs} which states that every neighbor of $u$ not in $P'$ is a leaf of $G$. Hence, the edges $e$ and $f$ are not parallel edges. Thus, $e$ and $f$ join $v$ to distinct vertices $u$ and $w$, respectively.

Recall that by our earlier assumption, $v_{e}$ is the (unique) neighbor of $v$ in $P'$. Theorem~\ref{observ-2-minimal-DPDP-graphs} implies that the vertices $u_{e}$ and $v_{f}$ belong to $D'$, while $u$, $w$ and $v_{f}$ belong to $P'$. If $d_H(u) \ge 3$, then noting that $u_e$ is a neighbor of $u$ not in $P'$ of degree~$2$, we contradict Theorem~\ref{observ-2-minimal-DPDP-graphs}. Hence, $d_H(u) = 2$. Analogously, $d_H(w) = 2$. Let $u'$ and $w'$ be the neighbor of $u$ and $w$, respectively, different from $v$ in $H$, and so $N_H(u) \setminus \{v\} = \{u'\}$ and $N_H(w)  \setminus \{v\} = \{w'\}$. Since $H \ne C_3$, we note that $w' \ne u$ (and $u'\ne w$). We remark that possibly, $u' = w'$. Since $\Delta(H)\ge 3$, at least one of the vertices $u'$ and $w'$ is not a leaf in $H$. By symmetry, we may assume that $u'$ is not a leaf in $H$, and so $d_H(u') \ge 2$. Proposition~\ref{prop:adjdeg2} with $x=v$, $y=u$, $x'=w$, and $y'=u'$ implies that $G$ is not a minimal $D\!P\!D\!P$-graph, the final contradiction which completes the proof of Theorem~\ref{propertis-of-minimal}.
\end{proof}

\medskip
We next provide a characterization of minimal $D\!P\!D\!P$-graphs in terms of good subgraphs. In the next theorem we prove that minimal $D\!P\!D\!P$-graphs are precisely 2-subdivision graphs of graphs that do have neither an isolated vertex nor a good subgraph.

\begin{thm} \label{S2(H) is minimal DPDP-graph} A graph $G$ is a minimal $D\!P\!D\!P$-graph if and only if $G=S_2(H)$, where $H$ is a graph that has neither an isolated vertex nor a good subgraph. \end{thm}

\begin{proof} Assume first that $G$ is a minimal $D\!P\!D\!P$-graph, and let $(D,P)$ be a~$D\!P$-pair in $G$. It follows from Theorem \ref{observ-2-minimal-DPDP-graphs} that  $G=S_2(H)$ for some graph $H$. Since no $D\!P\!D\!P$-graph has an isolated vertex, neither $S_2(H)$ nor $H$ has an isolated vertex. We now claim that $H$ has no good subgraph. Suppose, to the contrary, that $Q$ is a good subgraph in $H$. By definition, there exist a set of edges $E$ (where $E_Q^- \subseteq E\subseteq E_H \setminus E_Q$) and an orientation $A_E$ of $E$ such that in the partially oriented graph $H(A_E)$ there exists a family of oriented paths ${\cal P}=\{P_x \colon x\in V_Q\}$ satisfying the properties (1)--(3) stated in the definition of a good subgraph.

We adopt the following notation: If $e$ is an edge belonging to $E$, $\varphi_H(e) =\{v,u\}$, $\xi(e)= \{v_e,u_e\}$, and $e_A=(v,u)$, then $v,v_e,u_e,u$ is the $4$-path corresponding to $e$ in $S_2(H)$, and we write $p_1(e)=v$, $p_2(e)=v_e$, $p_3(e)=u_e$, and $p_4(e)=u$. If $e$ is a loop belonging to $E$, $\varphi_H(e) =\{v\}$, $\xi(e)= \{v_e^1,v_e^2\}$, then $v,v_e^1,v_e^2,v$ is the $3$-cycle corresponding to $e$ in $S_2(H)$, and we write $p_1(e)=v$, $p_2(e)=v_e^1$, $p_3(e)=v_e^2$, and $p_4(e)=v$. Finally, we denote by $e(P_x)$ the edge in $E$ corresponding to the last arc (or loop) in the oriented path $P_x\in {\cal P}$.

Let us consider now the spanning subgraph $G'$ of $G=S_2(H)$ in which
\[
E_{G'}= E_{S_2(H)}  \setminus \left(\bigcup_{e\in E_Q}\{ xy\colon \xi(e)=\{x,y\}\} \cup \{p_3\big(e(P_x)\big)p_4\big(e(P_x)\big)\colon P_x\in {\cal P}\}\right).
\]
More~intuitively, $G'$ is the graph obtained from $S_2(H)$ by removing the middle edge from the $4$-path corresponding to each edge of $Q$, and the third edge from the $4$-path corresponding to the last arc in every path $P_x\in {\cal P}$. A graph $H$, its $2$-subdivision graph $S_2(H)$, and the subgraph $G'$ of $S_2(H)$ corresponding to a good subgraph $Q$ in $H$ (drawn in bold) and a family of oriented paths ${\cal P}= \{P_x \colon x\in V_Q\}$ are shown in Fig. \ref{graphs-for-proofs}. Formally, $H$, $S_2(H)$, and $G'$ are the underlying graphs of the graphs in Fig.~\ref{graphs-for-proofs}.

We note that the sets
\[
D'= V_{H_0} \cup \{p_3(e_A)\colon e_A\in A_E\} \cup \bigcup_{e\in E_Q}\xi(e)
\]
and
\[
P'= \bigcup_{e_A\in A_E}\{p_1(e_A), p_2(e_A)\} \cup \bigcup_{e\in E_{H_0}}\xi(e)
\]
form a partition of the vertex set of $G'$. We now claim that $(D',P')$ is a~$D\!P$-pair in $G'$. If $V_{H_0}\ne \emptyset$, then it follows from the construction of $G'$ that $G'[V_{H_0}\cup \bigcup_{e\in E_{H_0}}\xi(e)] =S_2(H_0)$ and therefore, as it follows from the proof of Proposition~\ref{prop-S2(H)-jest-DPDP-grafem}, the pair $(V_{H_0},\bigcup_{e\in E_{H_0}}\xi(e))$ is a~$D\!P$-pair in $S_2(H_0)$. Thus, it remains to prove that the sets $D'' = D' \setminus V_{H_0}$ and $P''=P' \setminus \bigcup_{e\in E_{H_0}}\xi(e)$ form a~$D\!P$-pair in $G''=G'-S_2(H_0)$.

We show firstly that $D''$ is a dominating set of $G''$. Let $x$ be an arbitrary vertex in $V_{G''} \setminus D''=P''$. Then either $x=p_2(e_A)$ or $x=p_1(e_A)$ for some $e_A\in A_E$. In the first case $x$ is adjacent to $p_3(e_A)\in D''$. Thus assume that $x=p_1(e_A)$ and $e_A\in A_E$. If  $x=p_1(e_A)\in V_Q$, then   there exists an edge $f$ in $Q$ incident with $x$, and therefore $x$ is adjacent to $v_f\in \xi(f)\subseteq D''$. Finally assume that $x=p_1(e_A)\not\in V_Q$. Now $e_A$ belongs to some oriented path $P_v\in {\cal P}$. Since $x=p_1(e_A)\not\in V_Q$, there exists an arc $f_A$ on $P_v$ such that $p_4(f_A)=x=p_1(e_A)$, and therefore $x$ is adjacent to $p_3(f_A)\in D''$. This proves that $D''$ is a dominating set of $G''$.

We show next that $P''$ is a dominating set of $G''$. Let $y$ be an arbitrary vertex in $V_{G''} \setminus P''=D''$. If $y=p_3(e_A)$ for some $e_A\in A_E$, then $y$ is adjacent to $p_2(e_A)\in P''$. Finally assume that $y\in \xi(e)$ for some $e\in E_Q$. Without loss generality, we may assume that $\varphi_H(e)=\{u,v\}$, $\xi(e)=\{v_e,u_e\}$, and $y=v_e$. Thus, $y$ is adjacent to $p_1(f_A)\in P''$ where $f_A$ is the first arc in the unique path $P_v \in {\cal P}$ starting at $v$. This implies that $P''$ is a dominating set of $G''$. In addition, $P''$ is a paired-dominating set of $G''$, as the edges $p_1(e_A)p_2(e_A)$, where $e_A\in A_E$, form a perfect matching in the subgraph induced by $P''$. This proves that $(D'',P'')$ is a~$D\!P$-pair in $G''$, and implies that $(D',P')$ is a~$D\!P$-pair in a~proper spanning subgraph $G'$ of $G$, contradicting the minimality of $G$.

Assume now that $H$ is a graph that has neither an isolated vertex nor a good subgraph.
By Proposition~\ref{prop-S2(H)-jest-DPDP-grafem}, the $2$-subdivision graph $G=S_2(H)$ of $H$ is a $D\!P\!D\!P$-graph. We claim that $G$ is a minimal $D\!P\!D\!P$-graph. Suppose, to the contrary, that $G$ is not a minimal $D\!P\!D\!P$-graph. Thus some proper spanning subgraph $G'$ of $G$ is a~minimal $D\!P\!D\!P$-graph, and it follows from Theorem \ref{observ-2-minimal-DPDP-graphs} that $G'$ is a $2$-subdivision graph of some graph $H'$, i.e., $G'=S_2(H')$.

Since $G'$ is a~proper spanning subgraph of $G$, the set $E_G \setminus E_{G'}$ (of the edges removed from $G$) is nonempty and it is the union of disjoint subsets $E_{nn}'= (E_{G} \setminus E_{G'})\cap E_{nn}$ and $E_{no}'= (E_{G} \setminus E_{G'}) \setminus E_{nn}$, where $E_{nn}$ is the set of edges of $G$ each of which joins two vertices in $\bigcup_{e\in E_H} \xi(e)$. It follows from the definition of the $2$-subdivision graph that if $xy\in E_{G} \setminus E_{G'}$, then both $x$ and $y$ are leaves in $G'$ if $xy\in E_{nn}'$ and at least one of the vertices $x$ and $y$ is a~leaf in $G'$ if $xy\in E_{no}'$,  and $\{x,y\}\cap N_G[L_G]= \emptyset$ (since $G'$ is a~$D\!P\!D\!P$-graph). This implies that $G'$ has two types of components: those which have at least one leaf belonging to the set $V_{S_2(H)}^n$,  and those in which no leaf belongs to $V_{S_2(H)}^n$. From this and from Observation \ref{pierwsze-wlasnosci-S2(H)}\,(6) (and Corollary~\ref{c:paths_cycles}) it follows that if $F$ is a~component of $G'$, then  $F=S_2(P_{k+1})$ for some $k\in [4]$ and $F$ has at most one strong support vertex if $L_F\cap V_{S_2(H)}^n \ne  \emptyset$ or $F$ is an induced subgraph of $G$ if $L_F\cap V_{S_2(H)}^n = \emptyset$.

Let $F_1,\ldots,F_\ell$ be that components of $G'$ for which $L_{F_i}\cap V_{S_2(H)}^n \ne  \emptyset$ where $i \in [\ell]$. From this and from the fact that $F_i= S_2(P_{k_i+1})$ is of diameter $3k_i+1$ it follows that exactly one support vertex of $F_i$ is a vertex of $H$, say $\{v^i\}= S_{F_i}\cap V_H$ for $i \in [\ell]$. Let $\overline{v}^i$ be the (unique) leaf farthest from $v^i$ in $F_i$, and let $\widetilde{v}^i$ be the only vertex in $N_G(\overline{v}^i) \setminus N_{G'}(\overline{v}^i) \subseteq V_H$. Let $\overline{P}_i$ be the $v^i-\overline{v}^i$ path in $F_i$, and let $\widetilde{P}_i$ be the $v^i-\widetilde{v}^i$ path obtained from $\overline{P}_i$ by adding $\widetilde{v}^i$ and the edge $\overline{v}^i\widetilde{v}^i$. Since $v^i, \widetilde{v}^i\in V_{G'}$ and $d_{G'}(v^i, \widetilde{v}^i)= 3k_i-1$ for some $k_i \in [4]$, we may assume that $\overline{P}_i$ is the path $v^i = x^0, x^1, \ldots, x^{3k_i-1}= \overline{v}^i$ and $\widetilde{P}_i$ is the path $v^i = x^0, x^1, \ldots, x^{3k_i-1} = \overline{v}^i,x^{3k_i}= \widetilde{v}^i$, where $x^0, x^3, \ldots, x^{3k_i} \in V_H$, while $x^{3j+1} = x_e^{3j}$ and $x^{3j+2}= x_e^{3j+3}$, where $e$ is an edge joining $x^{3j}$ and $x^{3j+3}$ in $H$ for $j \in \{0\} \cup [k_i-1]$ (or  $x^{3j+1}= x_e^{3j\, 1}$ and $x^{3j+2}= x_e^{3j\,2}$ if $e$ is a loop at $x^{3j}$ and $j=k_i-1$). Now let $P_i$ be the oriented path $(x^0,a(x^0,x^3), x^3,\ldots, x^{3k_i-3}, a(x^{3k_i-3},x^{3k_i}), x^{3k_i})$ in $H$, where $a(x^{3j},x^{3j+3})$ is the arc which goes from $x^{3j}$ to $x^{3j+3}$ and which corresponds to the path $(x^{3j},x^{3j+1},x^{3j+2},x^{3j+3})$ in the path $\widetilde{P}_i$ for $j \in \{0\} \cup [k_i-1]$.

Let $Q=(V_Q,E_Q)$ be the subgraph of $H$, where $V_Q$ consists of those vertices of $H$ which are support vertices in $F_1,\ldots, F_\ell$, that is, $V_Q=\{v^1,v^2,\ldots,v^\ell\}$, and $E_Q$ consists of those edges (and loops) of $H$ whose middle edges were removed in the process of forming $G'$ from $G$, i.e., $E_Q= \{e\in E_H\colon \xi(e)=\{x,y\} \,\,{\rm and}\,\, xy\in E_{nn}'\}$ (see Fig. \ref{graphs-for-proofs}, where $Q$ (defined by $G'$) is the bold subgraph of the underlying graph of $H$). All that remains to prove is that $Q$ is a good subgraph in $H$.

Since the paths $\widetilde{P}_1, \ldots, \widetilde{P}_\ell$ are edge-disjoint in $G'$, it follows from the definition of $P_1, \ldots, P_\ell$ that ${\cal P}= \{P_1, \ldots, P_\ell\}$  is a family of arc-disjoint (not necessarily vertex-disjoint) oriented paths (in $H$) indexed by the vertices of $Q$. In addition, $P_i$ is the only path belonging to ${\cal P}$ and growing out from the vertex $v^i\in V_Q$, implying that $d^+_H(v^i)=1$ and $d^-_H(v^i)= d_H(v^i)-d_Q(v^i)-1$ for $i \in [\ell]$. From the same fact it follows that if the paths $P_i, P_j\in {\cal P}$, where $i \ne j$, are not vertex-disjoint, then the end vertex of (at least) one of them is the only vertex belonging to the second one. Consequently, if $x$ is a non-end vertex of a path $P_i\in {\cal P}$, then $d_H^+(x)=1$ (and $d^-_H(x)=d_H(x)-1$). Finally assume that $y$ is an end vertex of a path $P_i\in {\cal P}$. If $d_H^-(y)\ge d_H(y)$, then $y$ would be an isolated vertex in a $D\!P\!D\!P$-graph $G'$, which is impossible. Therefore, $d_H^-(y)<d_H(y)$. This proves that $Q$ is a good subgraph in $H$ and this completes the proof of Theorem~\ref{S2(H) is minimal DPDP-graph}.
\end{proof}

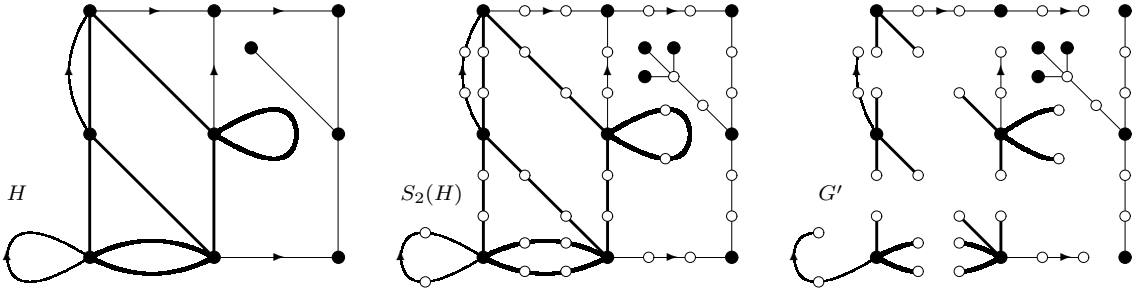
\begin{figure}[h!]\begin{center}  \special{em:linewidth 0.4pt} \unitlength 0.3ex \linethickness{0.4pt} \begin{picture}(340,80)
\put(-15,-5){\put(20,10){{\Thicklines\bezier{300}(20,10)(35,2)(50,10)
\bezier{300}(20,10)(35,18)(50,10)\path(20,10)(20,70)\path(20,40)(50,10)(50,40)(20,70)
\bezier{300}(50,40)(69,52)(70,40)\bezier{300}(50,40)(69,28)(70,40)}%
\path(80,10)(80,70){\bezier{300}(20,40)(10,55)(20,70)\put(15,57){\vector(0,1){0.0}} \path(20,70)(80,70)\path(50,40)(50,70)\path(50,10)(80,10)\put(37,70){\vector(1,0){0.0}} \put(67,70){\vector(1,0){0.0}}\put(67,10){\vector(1,0){0.0}}\put(50,57){\vector(0,1){0.0}}
\bezier{300}(20,10)(1,-2)(0,10)\bezier{300}(20,10)(1,22)(0,10)\put(0,12){\vector(0,1){0.0}}}
\multiput(20,10)(30,0){3}{\circle*{3}}\multiput(20,40)(30,0){3}{\circle*{3}}
\multiput(20,70)(30,0){3}{\circle*{3}}\path(59,61)(80,40)\put(59,61){\circle*{3}}\put(0,25){${}_{H}$}}}
\put(80,-5){\put(20,10){{\Thicklines\bezier{300}(20,10)(35,2)(50,10)
\bezier{300}(20,10)(35,18)(50,10)\path(20,10)(20,70)\path(20,40)(50,10)(50,40)(20,70)
\bezier{300}(50,40)(69,52)(70,40)\bezier{300}(50,40)(69,28)(70,40)}%
\path(80,10)(80,70){\bezier{300}(20,40)(10,55)(20,70)\put(15,57){\vector(0,1){0.0}} \path(20,70)(80,70)\path(50,40)(50,70)\path(50,10)(80,10)\put(37,70){\vector(1,0){0.0}} \put(67,70){\vector(1,0){0.0}}\put(67,10){\vector(1,0){0.0}}\put(50,57){\vector(0,1){0.0}}
\bezier{300}(20,10)(1,-2)(0,10)\bezier{300}(20,10)(1,22)(0,10)\put(0,12){\vector(0,1){0.0}}}
\multiput(20,10)(30,0){3}{\circle*{3}}\multiput(20,40)(30,0){3}{\circle*{3}}
\multiput(20,70)(30,0){3}{\circle*{3}}\path(59,61)(80,40)\put(59,61){\circle*{3}}
\path(59,54)(66,54)(66,61)\multiput(59,54)(7,7){2}{\circle*{3}}
\multiput(64,34)(0,11.8){2}{\whiten\circle{2.5}}\multiput(20,20)(0,10){2}{\whiten\circle{2.5}} \multiput(20,50)(0,10){2}{\whiten\circle{2.5}}\multiput(50,20)(0,10){2}{\whiten\circle{2.5}}
\multiput(50,50)(0,10){2}{\whiten\circle{2.5}}\multiput(80,20)(0,10){2}{\whiten\circle{2.5}}
\multiput(80,50)(0,10){2}{\whiten\circle{2.5}}\multiput(15.6,50)(0,10){2}{\whiten\circle{2.5}}
\multiput(6,4)(0,12){2}{\whiten\circle{2.5}}\multiput(30,6.6)(10,0){2}{\whiten\circle{2.5}}
\multiput(30,13.5)(10,0){2}{\whiten\circle{2.5}}\multiput(30,30)(10,-10){2}{\whiten\circle{2.5}}
\multiput(30,60)(10,-10){2}{\whiten\circle{2.5}}\multiput(30,70)(10,0){2}{\whiten\circle{2.5}}
\multiput(60,70)(10,0){2}{\whiten\circle{2.5}}\multiput(60,10)(10,0){2}{\whiten\circle{2.5}}
\multiput(73,47)(-7,7){2}{\whiten\circle{2.5}}\put(0,25){${}_{S_2(H)}$}}}
\put(175,-5){\put(20,10){{\Thicklines
\bezier{200}(20,10)(25,7)(30,6.3)\bezier{200}(40,6.3)(45,7)(50,10)
\bezier{200}(20,10)(25,12.9)(30,13.6)\bezier{200}(40,13.6)(45,12.9)(50,10)
\path(20,10)(20,20)\path(20,30)(20,50)\path(20,60)(20,70)\path(20,40)(30,30)
\path(40,20)(50,10)(50,20)\path(50,30)(50,40)(40,50)\path(30,60)(20,70)
\bezier{100}(50,40)(60,47)(64,45.8)\bezier{100}(50,40)(60,33)(64,34)}%
\path(80,10)(80,70){\bezier{300}(20,40)(13.7,50)(15.6,60)\put(15.2,57){\vector(0,1){0.0}}
\path(20,70)(70,70)\path(50,40)(50,60)\path(50,10)(70,10)\put(37,70){\vector(1,0){0.0}}
\put(67,70){\vector(1,0){0.0}}\put(67,10){\vector(1,0){0.0}}\put(50,57){\vector(0,1){0.0}}
\bezier{300}(20,10)(1,-2)(0,10)\bezier{300}(0,10)(0,16)(6,16)\put(0,12){\vector(0,1){0.0}}}
\multiput(20,10)(30,0){3}{\circle*{3}}\multiput(20,40)(30,0){3}{\circle*{3}}
\multiput(20,70)(30,0){3}{\circle*{3}}\path(59,61)(80,40)\put(59,61){\circle*{3}}
\path(59,54)(66,54)(66,61)\multiput(59,54)(7,7){2}{\circle*{3}}
\multiput(64,34)(0,11.8){2}{\whiten\circle{2.5}}
\multiput(20,20)(0,10){2}{\whiten\circle{2.5}}\multiput(20,50)(0,10){2}{\whiten\circle{2.5}}
\multiput(50,20)(0,10){2}{\whiten\circle{2.5}}\multiput(50,50)(0,10){2}{\whiten\circle{2.5}}
\multiput(80,20)(0,10){2}{\whiten\circle{2.5}}\multiput(80,50)(0,10){2}{\whiten\circle{2.5}}
\multiput(15.6,50)(0,10){2}{\whiten\circle{2.5}}\multiput(6,4)(0,12){2}{\whiten\circle{2.5}}
\multiput(30,6.6)(10,0){2}{\whiten\circle{2.5}}\multiput(30,13.5)(10,0){2}{\whiten\circle{2.5}}
\multiput(30,30)(10,-10){2}{\whiten\circle{2.5}}\multiput(30,60)(10,-10){2}{\whiten\circle{2.5}}
\multiput(30,70)(10,0){2}{\whiten\circle{2.5}}\multiput(60,70)(10,0){2}{\whiten\circle{2.5}}
\multiput(60,10)(10,0){2}{\whiten\circle{2.5}}\multiput(73,47)(-7,7){2}{\whiten\circle{2.5}}
\put(6,25){${}_{G'}$}} } \end{picture}\caption{Graphs $H$, $S_2(H)$, and a minimal spanning $D\!P\!D\!P$-subgraph $G'$ of $S_2(H)$} \label{graphs-for-proofs}\end{center}\end{figure}

\medskip
We are now in a position to present a proof of our main result, namely Theorem~\ref{thm:main1}. Recall its statement. \vskip 0.25 cm

\noindent \textbf{Theorem~\ref{thm:main1}}. \emph{If $G$ is a connected graph of order at least three, then the following statements are equivalent:
\begin{enumerate}
\item[$(1)$] $G$ a minimal $D\!P\!D\!P$-graph.
\item[$(2)$] $G=S_2(H)$ for some connected graph $H$, and either $(V_{S_2(H)}^o,V_{S_2(H)}^n)$ is the unique $D\!P$-pair in $G$ or $G$ is a cycle of length 3, 6 or 9.
\item[$(3)$] $G=S_2(H)$ for some connected graph $H$ that has neither an isolated vertex nor a good subgraph.
\item[$(4)$] $G=S_2(H)$ for some connected graph $H$ and no proper spanning subgraph of $G$ without isolated vertices is a $2$-subdivided graph.
\end{enumerate}
}
\begin{proof}
The statements (1), (2), and (3) are equivalent by Theorems \ref{propertis-of-minimal} and \ref{S2(H) is minimal DPDP-graph}. We shall prove that (1) and (4) are equivalent.

Assume that $G$ is a minimal $D\!P\!D\!P$-graph.
By Theorem \ref{observ-2-minimal-DPDP-graphs}, $G=S_2(H)$ for some connected graph $H$. In addition, since $G$ is a minimal $D\!P\!D\!P$-graph, no proper spanning subgraph of $G$ is a $D\!P\!D\!P$-graph. Thus no proper spanning subgraph of $G$ having no isolated vertex is a $2$-subdivided graph, as, by Proposition~\ref{prop-S2(H)-jest-DPDP-grafem}, every $2$-subdivided graph of a graph with no isolated vertex is a $D\!P\!D\!P$-graph. This proves the implication $(1)\Rightarrow (4)$.

If $G=S_2(H)$ for some connected graph $H$, then $G$ is a $D\!P\!D\!P$-graph (by Proposition~\ref{prop-S2(H)-jest-DPDP-grafem}). Assume that no proper spanning subgraph of $G$ without isolated vertices is a $2$-subdivided graph. We claim that $G$ is a minimal $D\!P\!D\!P$-graph. Suppose, to the contrary, that $G$ is not a minimal $D\!P\!D\!P$-graph. Then, since $G$ is a $D\!P\!D\!P$-graph, some proper spanning subgraph $G'$ of $G$ is a minimal $D\!P\!D\!P$-graph. Consequently, $G'$ has no isolated vertex (as no $D\!P\!D\!P$-graph has an isolated vertex). In addition, from the minimality of $G'$ and from Theorem \ref{observ-2-minimal-DPDP-graphs} it follows that $G'$ is a~$2$-subdivided graph. But this contradicts the statement (4) and proves the implication $(4)\Rightarrow (1)$.
\end{proof}

\medskip
The \emph{corona} $F \circ K_1$ of a graph $F$ is the graph obtained from $F$ by adding a pendant edge to each vertex of $F$. A \emph{corona graph} is a graph obtained from a graph $F$ by attaching any number of pendant edges to each vertex of $F$. In particular, the corona $F \circ K_1$ of a graph $F$ is a  corona graph.

\begin{cor} \label{S2-of-corona-is-a-minimal}
If $H$ is a corona graph, then its $2$-subdivision graph $S_2(H)$ is a~minimal $D\!P\!D\!P$-graph. In particular, $S_2(F \circ K_1)$ is a minimal $D\!P\!D\!P$-graph for every graph~$F$.
\end{cor}
\begin{proof}  Since every vertex of a corona graph is a leaf or a support vertex, it follows from Observation \ref{observation-leaves-in-paskudny-subgraphs} that $H$ has no good subgraph, and, therefore, $S_2(H)$ is a~minimal $D\!P\!D\!P$-graph, by Theorem \ref{S2(H) is minimal DPDP-graph}. \end{proof}

\begin{cor} \label{S2-of-S2-is-not-a-minimal} If $H$ is a connected graph, then $S_2(S_2(H))$ is a minimal $D\!P\!D\!P$-graph if and only if $H$ has either  exactly one edge or exactly one loop. \end{cor}

\begin{proof} If $E_H=\emptyset$, then $H$ consists of an isolated vertex, and
$S_2(S_2(H))= S_2(H)=H$ is not a~$D\!P\!D\!P$-graph. If $|E_H|= 1$, then $H= P_2$ (or $H=C_1$, resp.), and $S_2(S_2(H))= P_{10}$ (or $S_2(S_2(H))=C_9$, resp.) is a~minimal $D\!P\!D\!P$-graph. Assume now that $|E_H|\ge 2$. Thus, $V_H \setminus L_H \ne \emptyset$. If $v\in V_H \setminus L_H$, then $|E_H(v)|\ge 2$ and we consider two cases. Assume first that there is a~loop $e$ in $E_H(v)$. In this case  the vertices $v^1_e$, $v^2_e$, and the edge $v^1_ev^2_e$ form a~good subgraph in $S_2(H)$.  Consequently, by Theorem  \ref{S2(H) is minimal DPDP-graph},  $S_2(S_2(H))$ is not a~minimal $D\!P\!D\!P$-graph. Assume now that $E_H(v)= \{e_1, \ldots, e_k\}$ where $k\ge 2$, and no loop belongs to $E_H(v)$. Then the vertices $v, v_{e_1},\ldots, v_{e_{k-1}}$, and the edges $vv_{e_1}, vv_{e_2},\ldots, vv_{e_{k-1}}$ form a~good subgraph in $S_2(H)$. From this and from Theorem  \ref{S2(H) is minimal DPDP-graph} it again follows that  $S_2(S_2(H))$ is not a~minimal $D\!P\!D\!P$-graph.  \end{proof}

\section{$D\!P\!D\!P$-trees}

In this section we study the $D\!P\!D\!P$-trees, minimal $D\!P\!D\!P$-trees, and good subgraphs in trees. We begin with the following characterization of $D\!P\!D\!P$-trees.

\begin{proposition} \label{a-DPDP-tree} A tree $T$ is a $D\!P\!D\!P$-tree if and only if $T$ is a spanning supergraph of a $2$-subdivision graph of a forest without isolated vertices and good subgraphs.
\end{proposition}
\begin{proof} If $H$ is a forest without isolated vertices, then the forest $S_2(H)$ is a $D\!P\!D\!P$-graph (by Proposition~\ref{prop-S2(H)-jest-DPDP-grafem}) and every spanning supergraph of $S_2(H)$ is a $D\!P\!D\!P$-graph. In particular, any tree which is a spanning supergraph of $S_2(H)$ is a $D\!P\!D\!P$-tree.

Assume now that a tree $T$ is a $D\!P\!D\!P$-graph. Let $R$ be a spanning minimal $D\!P\!D\!P$-subgraph of $T$. Then $R$ is a forest and it follows from Theorems \ref{observ-2-minimal-DPDP-graphs}\,(4) and \ref{S2(H) is minimal DPDP-graph}
that $R=S_2(F)$ for some forest $F$ (without isolated vertices and good subgraphs) and therefore $T$ is a spanning supergraph of $S_2(F)$.
\end{proof}

We are interested in recognizing the structure of trees having a good subgraph. First we remark that each good subgraph in a tree is a forest. The following result shows if a tree has a good forest, then it also has a good subtree.

\begin{proposition} \label{corollary-warm-forest-in-tree} A tree has a good subgraph if and only if it has a good subtree.  \end{proposition}
\begin{proof} Assume that a forest $Q$ is a good subgraph in a tree $H$. Let $Q_1,\dots,Q_k$ ($k\ge 2$) be the components of $Q$. It suffices to prove that one of the components $Q_1,\dots,Q_k$ is a good subgraph in $H$. Let  $\mathcal{P} = \{P_v \colon v \in V_Q\}$ be a family of oriented paths indexed by the vertices of $Q$ and having the properties (1)--(3) stated in the definition of a~good subgraph (for some subset $E$, where $E_Q^-\subseteq E \subseteq E_H \setminus E_Q$, and some orientation $A_E$ of the edges in $E$). Let $\mathcal{P}_i$ denote the family $\{P_v \colon v \in V_{Q_i}\}$ where $i \in [k]$. From the properties of ${\cal P}$ and from the fact that $H$ is a tree it follows that ${\cal P}_i$ is a family of vertex-disjoint paths, each vertex of $Q_i$ is the initial vertex of exactly one path belonging to ${\cal P}_i$, and no path $P_v\in {\cal P}_i$ terminates at a~vertex of $Q_i$ or at a leaf of $H$. (Although, this time a path belonging to ${\cal P}_i$ can terminate at a vertex belonging to $Q_j$ or to a path in ${\cal P}_j$, $j\ne  i$.) However, from the same facts it follows that there exists a subtree $Q_{i_0}\in \{Q_1,\dots,Q_k\}$ such that no
path $P_v\in \bigcup_{j\ne i_0} {\cal P}_j$ terminates at $Q_{i_0}$.%
Now $Q_{i_0}$ is a~good subtree in $H$ as the family ${\cal P}_{i_0}$ has the properties (1)--(3) stated in the definition of a good subgraph (for the partially ordered graph $H[A_{E_{i_0}}]$, where $A_{E_{i_0}}$ is the set of arcs belonging to $A_E$ and covered by the paths of ${\cal P}_{i_0}$, see $Q_2$ or $Q_5$ in Fig.~\ref{good forest-graphs}).
\end{proof}

\begin{figure}[h!]\begin{center}  \special{em:linewidth 0.4pt} \unitlength 0.5ex \linethickness{0.4pt} \begin{picture}(120,45)\put(5,40){${}_{H}$}
\put(13,5){\put(0,0){\Thicklines\path(0,0)(10,0)\put(0,0){\circle*{2}}\put(10,0){\circle*{2}}}
\put(30,0){\Thicklines\path(0,0)(10,0)\put(0,0){\circle*{2}}\put(10,0){\circle*{2}}}
\put(50,0){\Thicklines\path(0,0)(10,0)\put(0,0){\circle*{2}}\put(10,0){\circle*{2}}}
\put(70,0){\Thicklines\path(0,0)(10,0)\put(0,0){\circle*{2}}\put(10,0){\circle*{2}}}
\put(0,20){\Thicklines\path(0,0)(10,0)\put(0,0){\circle*{2}}\put(10,0){\circle*{2}}}
\put(10,20){\Thicklines\path(0,0)(10,0)\put(0,0){\circle*{2}}\put(10,0){\circle*{2}}}
\put(70,0){\Thicklines\path(0,0)(0,10)\put(0,0){\circle*{2}}\put(10,0){\circle*{2}}}\path(0,0)(0,40)
\put(80,0){\Thicklines\path(0,0)(0,10)\put(0,0){\circle*{2}}\put(10,0){\circle*{2}}}\path(0,0)(100,0)
\put(0,0){\path(0,0)(0,20)\put(0,10){\circle*{2}}\put(0,20){\circle*{2}}\put(0,7){\vector(0,1){0.0}}}
\put(0,20){\path(0,0)(0,20)\put(0,10){\circle*{2}}\put(0,20){\circle*{2}}\put(0,7){\vector(0,1){0.0}}}
\put(10,20){\path(0,0)(0,20)\put(0,10){\circle*{2}}\put(0,20){\circle*{2}}\put(0,7){\vector(0,1){0.0}}}
\put(20,20){\path(0,0)(0,20)\put(0,10){\circle*{2}}\put(0,20){\circle*{2}}\put(0,7){\vector(0,1){0.0}}}
\put(30,0){\path(0,0)(0,20)\put(0,0){\circle*{2}}\put(0,10){\circle*{2}}\put(0,20){\circle*{2}}
\put(0,7){\vector(0,1){0.0}}}\put(40,0){\path(0,0)(0,20)\put(0,0){\circle*{2}}\put(0,10){\circle*{2}}
\put(0,20){\circle*{2}}\put(0,7){\vector(0,1){0.0}}}\put(60,0){\path(0,0)(0,20)\put(0,0){\circle*{2}}
\put(0,10){\circle*{2}}\put(0,20){\circle*{2}}\put(0,7){\vector(0,1){0.0}}}\put(70,10){\path(0,0)(0,20)
\put(0,0){\circle*{2}}\put(0,10){\circle*{2}}\put(0,20){\circle*{2}}\put(0,7){\vector(0,1){0.0}}}
\put(80,10){\path(0,0)(0,20)\put(0,0){\circle*{2}}\put(0,10){\circle*{2}}\put(0,20){\circle*{2}}
\put(0,7){\vector(0,1){0.0}}}\put(100,0){\circle*{2}}\put(20,0){\circle*{2}}\put(50,0){\circle*{2}}
\put(0,17){\vector(0,1){0.0}}\put(16,0){\vector(1,0){0.0}}\put(26,0){\vector(1,0){0.0}}\put(86,0){\vector(1,0){0.0}}
\put(44,0){\vector(-1,0){0.0}}\put(64,0){\vector(-1,0){0.0}}\put(5,0){\ellipse{16}{6}}
\put(35,0){\ellipse{16}{6}}\put(55,0){\ellipse{16}{6}}\put(10,20){\ellipse{26}{6}}\put(75,5){\ellipse{18}{18}}
\put(-5,16){${}_{Q_1}$}\put(-3,-5){${}_{Q_2}$}\put(27,-5){${}_{Q_3}$}\put(47,-5){${}_{Q_4}$}\put(80,-5){${}_{Q_5}$}}%
\end{picture} \caption{A good forest in a tree} \label{good forest-graphs}\end{center}\end{figure}
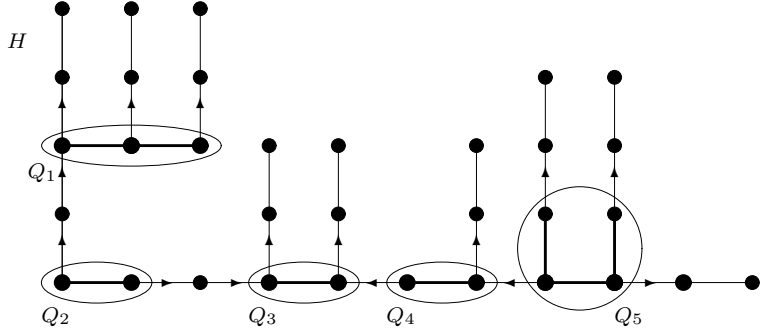

We observe that every tree can be a good subtree in a tree. The following result describes the place of a good subtree in a tree and connections between this good subtree and the rest of the tree.

\begin{proposition} \label{observation-warm-tree-in-tree} A tree $Q$ is a good subgraph in a tree $H$ if and only if no leaf of $H$ is a neighbor of $Q$ and the subgraph of $H$ induced by the set $N_H[V_Q]$ is a corona graph, that is, if and only if $N_H[V_Q] \cap L_H= \emptyset$ and $H[N_H[V_Q]]=Q \circ K_1$.
\end{proposition}
\begin{proof} Let $Q$ be a good subgraph of $H$ and let $\mathcal{P} = \{P_v \colon v \in V_Q\}$ be a family of oriented paths indexed by the vertices of $Q$ and having the properties (1)--(3) stated in the definition of a good subgraph (for some subset $E$, where $E_Q^-\subseteq E \subseteq E_H \setminus E_Q$, and some orientation $A_E$ of the edges in $E$). From these properties and from the fact that $H$ is a tree it follows that ${\cal P}$ is a family of vertex-disjoint paths, each vertex of $Q$ is the initial vertex of exactly one path belonging to ${\cal P}$, and no path $P_v\in {\cal P}$ terminates at a vertex of $Q$ or at a leaf of $H$. This proves that $N_H[V_Q] \cap L_H= \emptyset$. (The same follows directly from Observation \ref{observation-leaves-in-paskudny-subgraphs}). In addition, every vertex $v$ of $Q$ is adjacent to exactly one vertex in $V_H \setminus V_Q$, say $s_v$, which is the terminal vertex of the first arc in $P_v$. Since $H$ is a tree, the set $\{s_v\colon v\in V_Q\}$ is independent and, consequently, the subgraph of $H$ induced by $V_Q\cup \{s_v\colon v\in V_Q\}$ ($=V_H[V_Q]$) is a~corona graph isomorphic to $Q \circ K_1$.

Now assume that $Q$ is a subtree of $H$ such that $N_H[V_Q] \cap L_H= \emptyset$ and $H[N_H[V_Q]]= Q \circ K_1$. For a vertex $v$ of $Q$, let $v_\ell$ denote the only vertex in $N_H(v) \setminus V_Q$. Since the edge set $E=\{vv_\ell \colon v\in V_Q\}$, the arc set $A_E=\{(v,v_\ell) \colon v\in V_Q\}$, and the family of oriented paths ${\cal P}=A_E$ have properties (1)--(3) of the definition of a good subgraph, we note that $Q$ is a good subgraph in $H$. 	
\end{proof}

\begin{cor} \label{ostatni-o-drzezwach} If $H$ is a tree of order at least two, then $S_2(H)$ is a $D\!P\!D\!P$-tree. In addition, the $D\!P\!D\!P$-tree $S_2(H)$ is not a minimal $D\!P\!D\!P$-tree if and only if there is a~tree $Q$ in $H-(L_H\cup S_H)$ such that $Q\circ K_1$ is a subtree in $H-L_H$ and $d_H(x)=d_Q(x)+1$ for each vertex $x$ of $Q$. \end{cor}

\section{Open problems}

We close this paper with the following list of open problems that we have yet to settle.
\begin{enumerate}
\item How difficult is it to recognize graphs having good subgraphs?
\item How difficult is it to recognize whether a given graph is a good subgraph in
a~graph?
\item How difficult is it to recognize whether a given tree has good subtree?
\item Provide an algorithm for the problem determining a good subgraph of a graph.
\item Since every graph without isolated vertices is homeomorphic to
a $D\!P\!D\!P$-graph, it would be interesting to find the smallest number of subdivisions of edges of a graph in order to obtain a $D\!P\!D\!P$-graph.
\end{enumerate}


\medskip
\begin{thebibliography}{99}

\bibitem{AK} V. Anusuya, R. Kala, A note on disjoint dominating sets in graphs, \emph{Int. J. Contemp. Math. Sci.} 7 (2012) 2099--2110.

\bibitem{BDGHHU} I. Broere, M. Dorfling, W. Goddard, J.H. Hattingh, M.A. Henning, E. Ungerer, Augmenting trees to have two disjoint total dominating sets, \emph{Bull. Inst. Combin. Appl.} 42 (2004) 12--18.

\bibitem{ChLZ16} G.~Chartrand, L.~Lesniak, P.~Zhang,
\emph{Graphs and Digraphs}. CRC Press, Boca Raton, 2016.

\bibitem{DDH} P. Delgado, W.J. Desormeaux, T.W. Haynes, Partitioning the vertices of a~graph into two total dominating sets, \emph{Quaest. Math.} 39 (2016) 863--873. 

\bibitem{DHH17} W.J.~Desormeaux, T.W.~Haynes, M.A.~Henning,
Partitioning the vertices of a~cubic graph into two total dominating sets,
\emph{Discrete Appl. Math.} 223 (2017) 52--63. 

\bibitem{DGHH} M. Dorfling, W. Goddard, J.H. Hattingh, M.A. Henning, Augmenting a graph of minimum degree 2 to have two disjoint total dominating sets, \emph{Discrete Math.} 300 (2005) 82–90. 

\bibitem{HaynesHenning2005} T.W. Haynes, M.A. Henning, Trees with two disjoint minimum independent dominating sets, \emph{Discrete Math.} 304 (2005) 69--78.

\bibitem{HHLMS} S.M. Hedetniemi, S.T. Hedetniemi, R.C. Laskar, L. Markus, P.J. Slater, Disjoint dominating sets in graphs, Proc. ICDM 2006, Ramanujan Mathematics Society \emph{Lect. Notes Ser.} 7 (2008) 87--100.

\bibitem{HT} P. Heggernes, J.A. Telle, Partitioning graphs into generalized dominating sets, \emph{Nordic J. Comput.} 5 (1988) 128--142.

\bibitem{HLR09} M.A.~Henning, Ch.~L\"owenstein, D.~Rautenbach, Remarks about disjoint dominating sets, \emph{Discrete Math.} 309 (2009) 6451--6458. 

\bibitem{HLR10+} M.A. Henning, C. L\"owenstein, D. Rautenbach, Partitioning a graph into a dominating set, a total dominating set, and something else, \emph{Discuss. Math. Graph Theory} 30 (2010) 563--574. 

\bibitem{HLR10} M.A. Henning, C. L\"owenstein, D. Rautenbach, An independent dominating set in the complement of a minimum dominating set of a tree, \emph{Appl. Math. Lett.} 23 (2010) 79--81. 

\bibitem{HLR10++} M.A. Henning, C. L\"owenstein, D. Rautenbach, J. Southey,  Disjoint dominating and total dominating sets in graphs, \emph{Discrete Appl. Math.} 158 (2010) 1615--1623.

\bibitem{HM18} M.A.~Henning, A.J.~Marcon,
Semitotal domination in graphs: Partition and algorithmic results,
\emph{Util. Math.} 106 (2018) 165--184.

\bibitem{HR13} M.~A.~Henning, D.~F.~Rall,
On graphs with disjoint dominating and $2$-dominating sets,
\emph{Discuss. Math. Graph Theory} 33  (2013) 139--146. 

\bibitem{HS08} M.A. Henning, J. Southey, A note on graphs with disjoint dominating and total dominating sets, \emph{Ars Combin.} 89 (2008) 159--162.

\bibitem{HS09} M.A. Henning, J. Southey, A characterization of graphs with disjoint dominating and total dominating sets, \emph{Quaest. Math.} 32 (2009) 119--129.

\bibitem{HY} M.A. Henning, A. Yeo, \emph{Total Domination in Graphs}, Springer Monographs
in Mathematics, Springer,  2013. 

\bibitem{KJ} E.M. Kiunisala, F.P. Jamil, On pairs of disjoint dominating sets in a graph, \emph{Int. J. Math. Anal.} 10 (2016) 623--637. 


\bibitem{KS} V.R. Kulli, S.C. Sigarkanti, Inverse domination in graphs, \emph{Nat. Acad. Sci. Lett.} 14 (1991) 473--475.

\bibitem{LR10} C. Lowenstein, D. Rautenbach, Pairs of disjoint dominating sets and
the minimum degree of graphs, \emph{Graphs Combin.} 26 (2010) 407--424. 

\bibitem{MTZ19} M. Miotk, J. Topp, P. Żyliński, Disjoint dominating and $2$-dominating sets in graphs, arXiv:1903.06129v1.

\bibitem{Ore} O. Ore, \emph{Theory of Graphs},   Amer. Math. Soc. Colloq. Publ. 38, Amer. Math. Soc., Providence, RI, 1962.

\bibitem{SH10a} J.~Southey, M.A.~Henning, Graphs with disjoint dominating and paired-dominating sets, \emph{Cent. Eur. J. Math.} 8 (2010) 459--467. 

\bibitem{SH11a} J.~Southey, M.A.~Henning, Dominating and total dominating partitions in cubic graphs, \emph{Cent. Eur. J. Math.} 9 (2011) 699--708. 

\bibitem{SH11b} J.~Southey, M.A.~Henning, A characterization of graphs with disjoint dominating and paired-dominating sets, \emph{J. Comb. Optim.} 22 (2011) 217--234.

\end{thebibliography}
\end{document}